\documentclass[11p,leqno]{amsart}
\usepackage{amsmath}
\usepackage{amsfonts}
\usepackage{amssymb}
\usepackage{graphicx}
\usepackage{color}
\newtheorem{theorem}{Theorem}[section]
\newtheorem*{theorem*}{Theorem}

\newtheorem{corollary}[theorem]{Corollary}
\newtheorem{proposition}{Proposition}[section]
\newtheorem{prop}{Proposition}[section]
\newtheorem{remark}[theorem]{Remark}

\def\qed{\hfill $\square$}

\def\Ric{\text{Ric}}

\def\C{\Bbb C}
\def\Z{\Bbb Z}
\def\R{\Bbb R}
\def\CP{{\Bbb C} P}
\def\Sph{\Bbb S}

\def\bA{\bar{A}}
\def\bB{\bar{B}}
\def\bC{\bar{C}}
\def\bD{\bar{D}}

\def\id{\operatorname{id}}
\def\Ric{\operatorname{Ric}}
\def\Rm{\operatorname{Rm}}

\newcommand{\de}{\mathrm{d}}

 \newcommand{\beq}{\begin{equation}}
 \newcommand{\eeq}{\end{equation}}
 \newcommand{\beqn}{\begin{eqnarray*}}
 \newcommand{\eeqn}{\end{eqnarray*}}
 \newcommand{\beqa}{\begin{eqnarray}}
 \newcommand{\eeqa}{\end{eqnarray}}
 \newcommand{\CR}{\nonumber \\}
\newcommand{\HP}{{\mathbb{H}P}}

\numberwithin{equation}{section}

\begin{document}
\title[Ancient solutions of Ricci flow]{\bf Ancient solutions of Ricci flow on spheres and generalized Hopf fibrations}

\author{Ioannis Bakas, Shengli Kong and Lei Ni}



%    General info
\date{}

\maketitle

\begin{abstract} Ancient solutions arise in the study of Ricci flow singularities.
Motivated by the work of Fateev on 3-dimensional ancient solutions we construct high dimensional ancient solutions to Ricci flow on spheres and complex projective spaces as well as the twistor spaces over a compact quaternion-K\"ahler manifold. Differing from Fateev's examples most of our examples  are non-collapsed. The construction of this paper,  different from the ad hoc anastz of Fateev, is  systematic, generalizing (as well as unifying)
the previous constructions of Einstein metrics by Bourguignon-Karcher, Jensen, and Ziller in the sense that the existence problem to a backward nonlinear parabolic PDE is reduced to the study of  nonlinear
ODE systems. The key step of solving the reduced nonlinear ODE system is to find suitable monotonic and conserved quantities under Ricci flow with symmetry. Besides supplying new possible singularity models for Ricci flow on spheres and projective spaces, our solutions are counter-examples to some folklore conjectures on ancient solutions of Ricci flow on compact manifolds.
 As a by-product, we infer that some nonstandard Einstein metrics on spheres and complex projective spaces are unstable fixed points of the Ricci flow.

\end{abstract}

\section{Introduction}

Ancient solutions to Ricci flow arise as singularity models in the formation of finite time singularities of the Ricci flow on compact manifolds. In dimension $2$, there exists the well-known example due to Fateev-Onofri-Zamolodchikov \cite{FOZ}, King \cite{King} and Rosenau \cite{Ro}, which is
often called sausage model.
In this paper we present examples of ancient solutions on spheres (as well as some other compact manifolds). The three dimensional example on $\Sph^3$ is originally due to Fateev \cite{Fa96}. Because of many interests  in this example, we add here details (in the appendix of this paper ) to make it more accessible to mathematicians. The higher dimensional generalization is obtained by varying the connection metric on a principle bundle over a K\"ahler-Einstein manifold or a quaternion-K\"ahler manifold with positive scalar curvature, or by variations of a special Riemannian submersion structure, then solving a nonlinear ODE system derived out of the Ricci flow equation.

Besides the fact that they provide the prototype for singularities of Ricci flow solutions, there are several other motivations to  construct ancient solutions to Ricci flow. The first comes from \cite{P2} in which Perelman constructed a non-collapsing  rotationally symmetric ancient solution on $\Sph^n$. (In fact the original example is constructed on $\Sph^3$. With the recent strong convergence result of B\"ohm and Wilking \cite{BW}, this can be easily adapted to  dimensions above three \cite{Chowetc3}.) A natural question is whether or not there exist ancient solutions other than  the Einstein (trivial) ones and the rotationally symmetric examples of Perelman, especially in view of the result of Daskalopoulos, Hamilton and Sesum \cite{Das} on 2-dimensional ancient solutions.
On the other hand,
 in \cite{Ni-anc} the third author proved that any type-I (according to Hamilton \cite{form} ancient solutions can be divided into type-I and II), non-collapsed, compact ancient solution, whose curvature operator lies inside a  pinching family (in the sense of B\"ohm and Wilking \cite{BW}), including in particular the one constructed from the cone of  $2$-positive curvature operators
as well as the one from the cone of positive complex sectional curvatures \cite{BS}, must be isometric to a spherical space form. This generalizes the corresponding  result of Hamilton \cite{form} in dimension two and that of Perelman \cite{P-private} in dimension three. Notice that Perelman's example is of positive curvature operator, non-collapsed, but of type-II.
Thus, a natural question  arises  whether or not the non-collapsing condition in the above mentioned classification result can be removed so that one still has the same assertion for type-I ancient solutions. The third motivation, which is related to Perelman's example as well as the classification result, is that there exists a speculation asserting that  a type-II ancient solution of dimension three must be isometric to the rotationally symmetric one constructed by Perelman. This was formulated on  page 389 of \cite{CLN} along with some other questions on ancient solutions. Recently this speculation has been confirmed in \cite{Das} for dimension two,
where, in fact, a complete classification was obtained.

Concerning  ancient solutions of Ricci flow, there also exist motivations from physics, as they describe trajectories of the renormalization
group equations of certain asymptotically free local quantum field theories in the ultra-violet regime (see, for instance, \cite{Sasha} for
an overview). In
particular, two-dimensional non-linear sigma models with target space a Riemannian manifold with metric $g$ are perturbatively renormalizable
and their beta-function coincides with the Ricci curvature tensor at one loop \cite{Polya, Friedan}. Thus, the Ricci flow describes
changes of the metric $g$ (viewed as generalized coupling) under changes of the logarithm of the world-sheet length scale in quantum theory
provided that the curvature is small. Two-dimensional quantum non-linear sigma models with positively curved target space, and their ancient solutions,
serve as toy models for exploring reliably, within perturbation theory, the high energy behavior of asymptotically
free quantum field theories in four space-time dimensions, such as Yang-Mills theories for strong interactions \cite{Gross, Politz}.

The example of Fateev \cite{Fa96}, which will be presented here in detail together with a thorough analysis of its properties,
provides a negative answer to the last two questions raised above in dimension three. In \cite{Fa96}, the solution was described by
a set of functions given with explicit formulae. It is by no means easy to check that these functions indeed give an ancient solution to Ricci flow. In Section 3 along with  long computations in Appendix we provide a detailed verification of Fateev's result. Besides that it gives rise to a very interesting family  of ancient solutions, it also provides useful insight for the construction of various examples in higher dimensions. In Section 4, for the sake of higher dimensional examples,  we first study a couple of special solutions derived from Fateev's example. The generalization to higher dimensions is done via some general formulations that include metric constructions, using the  connection metrics on principle bundles, as well as the variational construction via a Riemannian submersion structure. In particular, the construction of higher dimensional examples via principle bundles is presented in Sections 5 and 6 (we summarize various useful formulations and collect the necessary formulae of curvature tensors in Section 2). The construction via the Riemannian submersion is slightly more general and is needed for one example on $\Sph^{15}$ and many other examples  on complex projective spaces and flag manifolds.  This is described in Section 7.

It turns out that the geometric part of our construction is very much related to the construction of Einstein metrics on spheres, complex projective spaces, and general homogenous spaces, a subject that has been extensively studied \cite{BK, J, WZ, Ziller1, Ziller2} (see also \cite{Besse} and references therein for a more complete list). Extra analytic part is needed here since we have to solve the Ricci flow equation, which is reduced to a nonlinear ODE system, instead of Einstein equation,  which is equivalent to a quadratic algebraic equation in this formulation. When solving the nonlinear system equivalent to the Ricci flow equation, instead of finding the solutions explicitly as in \cite{Fa96}, we prove the global existence by finding a first integral, since  the ODE system involved in the high dimensional case of this paper seems a bit more complicated than Fateev's three dimensional case.  Here,  we merely focus on solutions on spheres and the total space of the generalized Hopf fibrations, even though some of the techniques can be adapted to other manifolds. We also study the geometric properties of these examples and show that  they provide   counter-examples to several  speculations in higher dimensions.

The main difference between our examples and Fateev's three dimensional example is that some of our ancient solutions are non-collapsed, while every one in Fateev's family of examples is collapsed. Sometimes this is useful. For example the speculation that any non-collapsed positively curved ancient solution  on spheres must be rotationally symmetric is refuted by one of our examples.
The existence of ample examples shows, in particular, that a higher dimensional classification of ancient solutions on spheres is in general a lot more complicated, if not impossible, than the surface case achieved in \cite{Das} (noting that for dimension two, any ancient solution to Ricci flow on closed manifolds resides on the 2-sphere, unless it is flat).

{\it Another interesting feature of our examples is that some of them   `connect' non-canonical Einstein metrics of lower entropy (in the sense of Perelman \cite{P-entropy}) to the round Einstein metric on spheres. In particular, it implies that these non-canonical Einstein metrics with lower entropy on spheres are unstable fixed points of the Ricci flow equation.}

 Here we note that there have been some studies on the stability of Einstein metrics as well as shrinking Ricci soltions, for example in \cite{CHI}, via entropy functional of Perelman.  Another set of examples  {\it evolve  a Hermitian metric into a K\"ahler metric on complex projective spaces. }In fact, they `connect' the non-canonical Einstein metrics (which is not K\"ahler) to the Fubini-Study metric on the complex projective spaces. Hence it also shows {\it  examples of unstable fixed points of Ricci flow equation on the complex projective spaces}. This applies similarly to the non-K\"ahler Einstein metric on the twistor space over any quaternion-K\"ahler manifold with positive scalar curvature. Hence as a consequence of the main theorems of this paper, we demonstrate many unstable Einstein metrics.
Moreover some other examples in this paper have the effect of collapsing the  fiber of a Riemannian submersion structure as  the solutions develop singularities. A similar picture has been described in the program of B\"ohm and Wilking on Ricci flow in high dimensions \cite{BW-pro}. (In a recent work \cite{SW}, among other things, Song and Weinkove also show that similar phenomenon happens for K\"ahler-Ricci flow.) In particular,  on  $\Sph^{15}$, the following result is a  consequence of the special cases of Theorem \ref{main1-u1}, Theorem \ref{qk-main1} and Theorem \ref{sb-main1}:
\begin{theorem} Besides the rotationally symmetric type-II ancient solution constructed by Perelman, on $\Sph^{15}$, there are at least five nontrivial (non-Einstein) type-I ancient solutions to Ricci flow. The first one is collapsed with positive curvature operator, which converges to the round metric as the time approaches to the singularity. The second and the third ones are non-collapsed, with positive sectional curvature,  each `connecting'  one of the two known nonstandard Einstein metrics (at $t=-\infty$) to the round metric  as the time approaches to the singularity. The fourth one `starts' with (at $t=-\infty$) a nonstandard Einstein metric and collapses the fiber sphere $\Sph^3$ in the generalized Hopf fibration $\Sph^3\to \Sph^{15}\to \HP^3$ as the time approaches to the singularity. The fifth ancient solution `starts' with another nonstandard  Einstein metric and collapses the fiber sphere $\Sph^7$ in the generalized Hopf fibration $\Sph^7\to \Sph^{15}\to \Sph^8$ as the time approaches to the singularity. Here `connecting' means as $t$ approaches to each of the two ends the re-scaled metric converges to metrics on both ends, and `starts' means as $t\to -\infty$ the re-scaled metric limits to the `starting' metric.
\end{theorem}
One consequence of the above result (as well as the general Theorems implying it) may be of significance is that {\it  $\Sph^3 \times \R^{4n}$ and $\Sph^7 \times \R^8$ can be the singularity models of Ricci flow on spheres.} Note here that the neck-pinching examples have been previously  constructed in \cite{Si, AK}.

 Here we call an ancient solution {\it collapsed} if one can not find a $\kappa>0$ such that it is $\kappa$-noncollapsed on all scales in the sense of Perelman \cite{P-entropy}. We refer the reader to Theorem \ref{main1-u1} in Section 5, Theorem \ref{qk-main1} in Section 6 and Theorem \ref{sb-main1} in Section 7 for much general theorems. These theorems particularly  imply that {\it there are at least one nontrivial type-I ancient solution on $\Sph^{2m+1}$, at least three nontrivial type-I ancient solutions on $\Sph^{4m+3}$ and two on $\CP^{2m+1}$. Theorem \ref{qk-main1} also implies similar result on the twistor spaces over quaternion-K\"ahler manifolds with positive scalar curvature.}

 It remains an interesting question if the present considerations can be modified appropriately to obtain noncompact examples of ancient solutions. It is also not known if one can extend Fateev's construction of type-II ancient solutions. As a modest classification question we propose that for the homogenous ancient solutions as well as  ones of the cohomogeneity one in view of the recent progresses in this direction on positively curved manifolds \cite{GWZ}. A more specific question is to classify the cohomogeneity one three dimensional ancient solutions. After the first draft of this paper, we were brought attention to the preprint \cite{CS}, where 3-dimensional ancient solutions to Ricci flow on homogeneous spaces are studied.

\section{Preliminaries}
\subsection{ }  First we recall several different ways of parametrizing  the sphere that are needed in the sequel. Letting $\Sph^p\subset \R^{p+1}$ and $\Sph^q \subset \R^{q+1}$, we construct $(0, \frac{\pi}{2})\times \Sph^p \times \Sph^q \to \Sph^{p+q+1}$ as
$$
(\theta, x, y)\to (x\cos \theta, y \sin \theta) \in \R^{p+q+2}.
$$
This map is not onto. But only two end spheres (corresponding to $\theta =0$ and $\theta =\frac{\pi}{2}$) need  to be added to have the whole $\Sph^{p+q+1}$.
Under this representation, the standard metric on $\Sph^{p+q+1}$ can be written as
\begin{equation}\label{warp1}
\de s^2 =\de \theta^2 +\cos^2 \theta \de\sigma_2^2 +\sin^2 \theta \de\sigma_1^2
\end{equation}
where $\de\sigma_2^2$ and $\de\sigma_1^2$ are the standard metrics on $\Sph^{p}$ and $\Sph^q$, respectively. Since this coordinate becomes singular as $\theta \to 0$ and $\theta \to \frac{\pi}{2}$, and we shall make use of the doubly warped product to construct a family of metrics on spheres, we need a result to ensure that the metric originally defined on
$(0, \frac{\pi}{2})\times \Sph^p \times \Sph^q$ can be extended smoothly to $\Sph^{p+q+1}$.

Let $(M^p, \de\sigma_2^2)$ be a compact Einstein manifold and $(\Sph^q, d\sigma^2_1)$ be the standard metric on the sphere. Consider the doubly warped product metric
\begin{equation}\label{BB1}
\de s^2 =\de \theta^2 +f^2(\theta) \de \sigma_1^2 +g^2(\theta)\de \sigma^2_2.
\end{equation}
The following proposition gives the condition to close up the metric at the `end'.
\begin{proposition}[B\'erard-Bergery]\label{reg1} Assume that $f$ and $g$ are smooth positive functions defined on $(0, \theta_0)$. Then, $ds^2$ extends to give a smooth positive definite metric in a
neighborhood of $\theta=0$ if and only if

(1) $f(\theta)$ extends smoothly to an odd function of $t$ with $f'(0)=1$;

(2) $g(\theta)$ extends smoothly to a strictly positive, even function of $\theta$.
\end{proposition}

Verifying this for the special case (\ref{warp1}),
confirms that $\de s^2$  can be extended to a smooth metric on $\Sph^{p+q+1}$ as $\theta \to 0$ and $\theta\to \frac{\pi}{2}$.

\subsection{} The sphere can also be viewed as the total space of a principle bundle. The prime example is the Hopf fibration $\mathsf{S}^1\to \Sph^{2m+1}\to \CP^m$, which is given by $(z_1, z_2, \cdots z_{m+1}) \in \C^{m+1}$ with $\sum |z_i|^2=1$ to the line $[z_1, \cdots, z_{m+1}]\in \CP^m$ and its generalization $\Sph^3 \to \Sph^{4m+3}\to \HP^m$, which is given similarly using the quaternions. It is convenient to set up the following general formulation. Let $\pi: ~P \longrightarrow M$ be a principle $G$-bundle over a
 Riemannian manifold $(M,g)$ with $\operatorname{dim} M = n$, $\operatorname{dim}  G =p$. Let
 $\sigma \in \Omega^1( P) \otimes \frak{g}$ be a connection on $P$
 and $\langle \cdot,\cdot\rangle_{\frak{g}}$ be a bi-invariant metric on $\frak{g}$. For any $a, b>0$, we define a metric
 $\widetilde{g}_{a,b}$ on the total space $P$ as follows:
 \[
 \widetilde{g}_{a, b} =a\, \langle \sigma(\cdot), \sigma(\cdot) \rangle_{\frak{g}}+ b\,  \pi^* g .
 \]
More precisely, for any $o\in P$ and any vectors $X, Y\in T_oP$, 
\[
\widetilde g(X, Y)=a\, \langle\sigma(X), \sigma(Y)\rangle_{\frak{g}}+b\, (\pi^* g)(X, Y).
\]
We need to  compute the curvature of this metric on the principle bundle. Let $1 \leq i,j,k \leq n, ~n+1 \leq \alpha, \beta, \gamma \leq n+p,~1
 \leq A, B, C \leq n+p$.  Let $\{ \sigma_1, \cdots, \sigma_n \}$ be an orthonormal frame of
 $\Omega^1(M)$. Then, Cartan's structure equations are given by
 \beqn \de \sigma_i &=& - \sigma_{ij} \wedge \sigma_j, \CR \de
 \sigma_{ij} + \sigma_{ik} \wedge \sigma_{kj} &=& \frac{1}{2}
 R_{ijkl} ~ \sigma_k \wedge \sigma_l. \eeqn Let $\{ X_{n+1}, \cdots,
 X_{n+p} \}$ be an orthonormal base of $\frak{g}$. The structure constants $C_{\alpha \beta}^{\gamma}$ are defined via
 \[
 [ X_{\alpha}, X_{\beta} ] = C_{\alpha \beta}^{\gamma} X_{\gamma}.
 \]
 Since the metric $\langle \cdot,\cdot\rangle_{\frak{g}}$ is biinvariant, $ C_{\alpha
 \beta}^{\gamma}$ are skew symmetric in every pair of indices. Define the orthonormal $1$-form $\sigma_\alpha$ by
 \[
 \sigma = \sigma_{\alpha} \cdot X_{\alpha}.
 \]
 Then
 \[
 \{ \pi^* \sigma_1, \cdots, \pi^* \sigma_n, \sigma_{n+1}, \cdots,
 \sigma_{n+p} \}
 \]
 is an orthonormal frame of $\Omega^1(P)$ for $\widetilde{g}_{1,1}$. For the sake of convenience,
  we also denote $ \pi^* \sigma_i$ by $\sigma_i$. The
 curvature of the connection $\sigma$ is defined as
 \[
 \Omega = \de \sigma + \frac{1}{2} [ \sigma, \sigma].
 \]
 If we write $ \Omega = \Omega^{\alpha} \cdot X_{\alpha}$, and
$$ \Omega^{\alpha} = \frac{1}{2} F_{ij}^{\alpha} \sigma_i \wedge
 \sigma_j ,
 $$
 then, we have
 \[
 \de \sigma_{\alpha} +  \frac{1}{2} C_{\beta \gamma}^{\alpha}
 \sigma_{\beta} \wedge \sigma_{\gamma} = \frac{1}{2} F^{\alpha}_{ij}
 \sigma_i \wedge \sigma_j.
 \]
The first covariant derivative of $F_{ij}^{\alpha}$ is defined by
 \[
 \sum_k F_{ij, k}^{\alpha} \sigma_k = \de F_{ij}^{\alpha} -
 F_{kj}^{\alpha} \sigma_{ki} - F_{ik}^{\alpha} \sigma_{kj}  -
 F_{ij}^{\beta} C^{\alpha}_{\beta \gamma } \sigma_{\gamma}
 \]
and the second  Bianchi identity asserts that
 \[
 F_{ij, k}^{\alpha} + F_{jk,i}^{\alpha} + F_{ki,j}^{\alpha} =0.
 \]
For general $\widetilde{g}_{a, b}$, let
$\widetilde{\sigma}_{\alpha} = \sqrt{a} \sigma_{\alpha}, ~\widetilde{\sigma}_i = \sqrt{b} \sigma_i$.
Then $\{ \widetilde{\sigma}_{\alpha}, \widetilde{\sigma}_i  \}$ is an orthonormal frame of $T^*P$.
 The Levi-Civita connection $1$-form $\widetilde{\sigma}_{AB}$ of $\widetilde{g}_{a, b}$ is uniquely determined
 by
 \beqn
 & & \widetilde{\sigma}_{AB} = - \widetilde{\sigma}_{BA}, \CR
 & & \de \widetilde{\sigma}_A = -\widetilde{\sigma}_{AB} \wedge
 \widetilde{\sigma}_B.
 \eeqn
 Therefore,
 \beqn
 \widetilde{\sigma}_{\alpha \beta} &=& - \frac{1}{2} C_{\beta \gamma}^{\alpha} \sigma_{\gamma}, \CR
 \widetilde{\sigma}_{\alpha i} &=& \frac{1}{2} \left( \frac{a}{b} \right)^{\frac{1}{2}} F_{ij}^{\alpha} \sigma_j,  \CR
 \widetilde{\sigma}_{ij} &=& \sigma_{ij} - \frac{a}{2b} F_{ij}^{\alpha} \sigma_{\alpha}.
 \eeqn

 The Riemannian curvatures of $\widetilde{g}_{a,b}$ can be computed using  the equation
 \[
 \de \widetilde{\sigma}_{AB} + \widetilde{\sigma}_{AC} \wedge
 \widetilde{\sigma}_{CB} = \frac{1}{2} \widetilde{R}_{ABCD} \widetilde{\sigma}_C \wedge
 \widetilde{\sigma}_D.
 \]
 Therefore,
 \beqn
\frac{1}{2} \widetilde{R}_{\alpha \beta
 ab} \widetilde{\sigma}_A \wedge \widetilde{\sigma}_B &=&
 \frac{1}{8a} C_{\alpha \beta}^{\gamma}
 C_{\delta \eta}^{\gamma}  \widetilde{\sigma}_{\delta} \wedge \widetilde{\sigma}_{\eta} - \frac{1}{4} \left( \frac{a}{b^2} F_{ik}^{\alpha} F_{jk}^{\beta} + \frac{1}{b}
 F_{ij}^{\gamma} C_{\alpha \beta}^{\gamma} \right)\widetilde{\sigma}_i \wedge \widetilde{\sigma}_j ,\CR
 \frac{1}{2} \widetilde{R}_{i \alpha AB} \widetilde{\sigma}_A \wedge \widetilde{\sigma}_B &=&
 \frac{1}{4} \left(
 \frac{a}{b^2} F_{ik}^{\beta} F_{jk}^{\alpha} - \frac{1}{b} F^{\gamma}_{ij} C_{\alpha \beta}^{\gamma}  \right) \widetilde{\sigma}_j \wedge \widetilde{\sigma}_{\beta}
 + \frac{1}{2} \left( \frac{a}{b^3} \right)^{\frac{1}{2}} F^{\alpha}_{ij,k} \widetilde{\sigma}_j \wedge \widetilde{\sigma}_k, \CR
\frac{1}{2} \widetilde{R}_{ijAB} \widetilde{\sigma}_a \wedge
\widetilde{\sigma}_b
 &=& \frac{1}{2b} R_{ijkl} \widetilde{\sigma}_k \wedge
\widetilde{\sigma}_l - \frac{a}{4b^2} \left( F_{ij}^{\alpha} F_{kl}^{\alpha} +
 F_{ik}^{\alpha}F_{jl}^{\alpha} \right) \widetilde{\sigma}_k \wedge
\widetilde{\sigma}_l  \CR & & -\frac{1}{2} \left( \frac{a}{b^3}
\right)^{\frac{1}{2}} F^{\alpha}_{ij,k} \widetilde{\sigma}_k \wedge
\widetilde{\sigma}_{\alpha} -\frac{1}{4} \left( \frac{a}{b^2} F_{ik}^{\alpha}
F_{jk}^{\beta} + \frac{1}{b}
 F_{ij}^{\gamma} C_{\alpha \beta}^{\gamma} \right)\widetilde{\sigma}_{\alpha} \wedge
 \widetilde{\sigma}_{\beta}.
 \eeqn
Hence, the Riemannian curvature of $\widetilde{g}_{a, b}$ is given in components by
 \begin{eqnarray}
 \widetilde{R}_{\alpha \beta \delta \eta } &=& \frac{1}{4a}  C^{\gamma}_{\alpha \beta} C_{\delta \eta}^{\gamma}, \nonumber\\
 \widetilde{R}_{\alpha \beta  \gamma i} &=& 0, \nonumber\\
 \widetilde{R}_{i\alpha j \beta} &=& - \frac{1}{4b} F^{\gamma}_{ij} C_{\alpha \beta}^{\gamma} +
 \frac{a}{4b^2} F_{ik}^{\beta} F_{jk}^{\alpha}, \nonumber\\
 \widetilde{R}_{ij \alpha \beta} &=& - \frac{1}{2b} F^{\gamma}_{ij} C_{\alpha \beta}^{\gamma} -
 \frac{a}{4b^2} \left(
 F_{ik}^{\alpha} F_{jk}^{\beta} - F_{ik}^{\beta}F_{jk}^{\alpha} \right), \label{principle-curvature}\\
 \widetilde{R}_{ijk \alpha} &=& - \frac{1}{2} \left( \frac{a}{b^3} \right)^{\frac{1}{2}} F^{\alpha}_{ij,k}, \nonumber\\
 \widetilde{R}_{ijkl} &=& \frac{1}{b} R_{ijkl} - \frac{a}{4b^2} \left( 2 F_{ij}^{\alpha} F_{kl}^{\alpha}
 + F_{ik}^{\alpha}F_{jl}^{\alpha} - F_{il}^{\alpha} F_{jk}^{\alpha} \right), \nonumber
 \end{eqnarray}
 and the Ricci curvature of $\widetilde{g}_{a, b}$ takes the form
 \begin{eqnarray}
 \widetilde{R}_{\alpha \beta} &=& \frac{1}{4a} C^{\gamma}_{\alpha \eta} C_{\beta
 \eta}^{\gamma} + \frac{a}{4b^2} F_{ij}^{\alpha} F_{ij}^{\beta} ,\nonumber\\
 \widetilde{R}_{i \alpha} &=& \frac{1}{2} \left( \frac{a}{b^3} \right)^{\frac{1}{2}} F^{\alpha}_{ij,j}, \label{principle-ricci}\\
 \widetilde{R}_{ij} &=& \frac{1}{b} R_{ij} - \frac{a}{2b^2} F_{ik}^{\alpha}
 F_{jk}^{\alpha}.\nonumber
 \end{eqnarray}
 The construction of Einstein metrics via the principle bundle were studied before.  For example it was used  in \cite{J}. Our formulae (\ref{principle-curvature}) and (\ref{principle-ricci}) above follow from the computations in \cite{J} with the appropriate modifications (cf. Proposition 5 of \cite{J}).

\subsection{}
Another formulation  needed for  our examples is  the Riemannian submersion, which also plays important role in the construction of Einstein metrics \cite{Besse}.  Let  $\pi:(P,g) \longrightarrow (M,
 \check{g})$ be a Riemannian submersion with $P, M$ being compact manifolds. Let $\hat{g}$ be the restriction of $g$ to the
 fibre. Let $\{e_{\alpha}, e_i \}$ be an orthonormal frame of tangent vector of $(P, g)$, where $\{e_{\alpha} \}$ is vertical and $\{ e_i
 \}$ is horizontal. Recall that \cite{Besse} the  O'Neill tensor $A$ and $T$ are defined by
 \begin{eqnarray*}
 & & A_{e_i} e_j = (\nabla_{e_i} e_j)^{\bot}, \quad \quad
 A_{e_i} e_{\alpha} = (\nabla_{e_i} e_{\alpha})^{\top}, \quad \quad
  A_{e_{\alpha}} e_i = 0, \quad \quad
  A_{e_{\alpha}} e_{\beta} = 0, \\
 & & T_{e_{\alpha}} e_i = (\nabla_{e_{\alpha}} e_i)^{\bot}, \quad \quad
  T_{e_{\alpha}} e_{\beta} = (\nabla_{e_{\alpha}} e_{\beta})^{\top}, \quad\quad
  T_{e_i} e_j = 0, \quad \quad
  T_{e_i} e_{\alpha} = 0,
 \end{eqnarray*}
 where $\bot$ and $\top$ represent the vertical and horizontal components respectively. The fiber is totally geodesic if and only if $T=0$. In this article,
 we only consider the Riemannian submersion with totally geodesic fibre. It is then known that all the fibers are isometric \cite{Besse}. O'Neill's
 formulae (see for example, Chapter 9, Section D of \cite{Besse}) gives the Riemannian curvature tensor of $g$ in terms of the curvature tensor of the base manifold, the curvature of the fiber and  the O'Neill's operator $A$  via the formulae below
 \begin{eqnarray*}
 & & R_{\alpha \beta \gamma \delta} = \hat{R}_{\alpha \beta \gamma \delta}, \quad
 \quad  R_{\alpha \beta \gamma i} = 0, \\
 & & R_{i \alpha j \beta} = \langle(\nabla_{e_{\alpha}} A)_{e_i} e_j , e_{\beta}\rangle+\langle A_{e_i} e_{\alpha}, A_{e_j} e_{\beta}\rangle, \CR
 & & R_{ij \alpha \beta} =\langle(\nabla_{e_{\alpha}} A)_{e_i} e_j , e_{\beta}\rangle - \langle(\nabla_{e_{\beta}} A)_{e_i} e_j , e_{\alpha}\rangle
 +\langle A_{e_i} e_{\alpha}, A_{e_j} e_{\beta}\rangle-\langle A_{e_i} e_{\beta}, A_{e_j} e_{\alpha}\rangle, \\
 & & R_{ijk \alpha}=  \langle(\nabla_{e_k} A)_{e_i} e_j, e_{\alpha}\rangle, \\
 & & R_{ijkl} = \check{R}_{ijkl} - 2\langle A_{e_i} e_j, A_{e_k} e_l\rangle -\langle A_{e_i} e_k, A_{e_j} e_l\rangle
 +\langle A_{e_i} e_l, A_{e_j} e_k\rangle ,
 \end{eqnarray*}
 whereas the Ricci curvature has the expression
 \beqn
 & & R_{\alpha \beta} = \hat{R}_{\alpha \beta} + \sum_i \langle A_{e_i} e_{\alpha}, A_{e_i} e_{\beta}\rangle, \CR
 & & R_{i \alpha} = \sum_j \langle (\nabla_{e_j} A)_{e_j} e_i , e_{\alpha}\rangle, \CR
 & & R_{ij} = \check{R}_{ij}- 2 \sum _k \langle  A_{e_i} e_k ,  A_{e_j} e_k\rangle.
 \eeqn
 Here $\hat{R}$ and $\check{R}$ are the curvature tensors of the fiber and the base, respectively.

 Now consider the variation $\widetilde{g}_{a, b} = a \hat{g} + b \check{g}$. Let
 $\widetilde{e}_{\alpha} = \frac{1}{\sqrt{a}} e_{\alpha}, ~\widetilde{e}_i = \frac{1}{\sqrt{b}}
 e_i.$
 Then, $\{\tilde{e}_{\alpha}, \tilde{e}_i \}$ is an orthonomal basis of $(P,
 \widetilde{g}_{a, b})$. Modifying the computations in Section G, Chapter 9 of \cite{Besse}, we  obtain the Riemannian curvature tensor of $\widetilde{g}_{a,b}$:
 \begin{eqnarray}
 & & \widetilde{R}_{\alpha \beta \gamma \delta} = \frac{1}{a} \hat{R}_{\alpha \beta \gamma \delta},
 \quad \quad  \quad \quad \quad \quad\widetilde{R}_{\alpha \beta \gamma i} = 0, \nonumber \CR
 & & \widetilde{R}_{i \alpha j \beta} =  \frac{1}{b} \langle(\nabla_{e_{\alpha}} A)_{e_i} e_j , e_{\beta}\rangle +\frac{1}{b}\langle A_{e_i} e_{\alpha}, A_{e_j} e_{\beta}\rangle
 - \frac{1}{b} \left( 1- \frac{a}{b} \right) \langle A_{e_i} e_{\beta}, A_{e_j} e_{\alpha}\rangle, \nonumber \\
 & & \widetilde{R}_{ij \alpha \beta} = \frac{1}{b} \langle (\nabla_{e_{\alpha}} A)_{e_i} e_j , e_{\beta}\rangle  - \frac{1}{b}\langle (\nabla_{e_{\beta}} A)_{e_i} e_j , e_{\alpha}\rangle +\frac{1}{b}\left( 2- \frac{a}{b} \right) \langle A_{e_i} e_{\alpha}, A_{e_j} e_{\beta}\rangle\nonumber \\
 & & \quad \quad \quad \quad- \frac{1}{b} \left( 2- \frac{a}{b} \right) \langle A_{e_i} e_{\beta}, A_{e_j} e_{\alpha}\rangle, \label{sm-curvature}\\
 & & \widetilde{R}_{ijk \alpha} = \left( \frac{a}{b^3} \right)^{\frac{1}{2}} \langle (\nabla_{e_k} A)_{e_i} e_j, e_{\alpha}\rangle,\nonumber  \\
 & & \widetilde{R}_{ijkl}= \frac{1}{b} \check{R}_{ijkl} - 2 \frac{a}{b^2} \langle A_{e_i} e_j, A_{e_k} e_l\rangle  - \frac{a}{b^2} \langle A_{e_i} e_k, A_{e_j} e_l\rangle + \frac{a}{b^2} \langle A_{e_i} e_l, A_{e_j} e_k\rangle. \nonumber
 \end{eqnarray}
 The Ricci curvature of $\widetilde{g}_{a, b}$ is given by
 \begin{eqnarray}
 & & \tilde{R}_{\alpha \beta} = \frac{1}{a} \hat{R}_{\alpha \beta} + \frac{a}{b^2} \sum_i \langle A_{e_i} e_{\alpha}, A_{e_i} e_{\beta}\rangle, \nonumber \\
 & & \tilde{R}_{i \alpha} = \left( \frac{a}{b^3} \right)^{\frac{1}{2}} \sum_j \langle (\nabla_{e_j} A)_{e_j} e_i , e_{\alpha}\rangle,  \label{sm-ricci} \\
 & & \tilde{R}_{ij} = \frac{1}{b} \check{R}_{ij}- 2 \frac{a}{b^2} \sum _k \langle  A_{e_i} e_k ,  A_{e_j}
 e_k\rangle.\nonumber
 \end{eqnarray}
 Hence the Riemannian curvature tensor and the Ricci curvature of $\widetilde{g}_{a, b}$ and $g$ are related by
 \begin{eqnarray}
 & & \widetilde{R}_{\alpha \beta \gamma \delta} = \frac{1}{a} R_{\alpha \beta \gamma \delta}, \quad \quad
  \widetilde{R}_{\alpha \beta \gamma i} = 0, \nonumber\\
 & & \widetilde{R}_{i \alpha j \beta} =  \frac{1}{b} R_{i \alpha j \beta} - \frac{1}{b} \left( 1- \frac{a}{b} \right) \langle A_{e_i} e_{\beta}, A_{e_j} e_{\alpha}\rangle, \CR
 & & \widetilde{R}_{ij \alpha \beta} = \frac{1}{b} R_{ij \alpha \beta} +\frac{1}{b} \left( 1- \frac{a}{b} \right) \langle A_{e_i} e_{\alpha}, A_{e_j} e_{\beta}\rangle
 - \frac{1}{b} \left( 1- \frac{a}{b} \right) \langle A_{e_i} e_{\beta}, A_{e_j} e_{\alpha}\rangle, \label{sb-curvatue-2}\\
 & &\widetilde{R}_{ijk \alpha} = \left( \frac{a}{b^3} \right)^{\frac{1}{2}} R_{ijk \alpha}, \CR
 & & \widetilde{R}_{ijkl}= \frac{a}{b^2} R_{ijkl} + \frac{1}{b} \left( 1- \frac{a}{b} \right)
 \check{R}_{ijkl},\nonumber
 \end{eqnarray}
 and
 \begin{eqnarray}
 & & \tilde{R}_{\alpha \beta} =  \frac{a}{b^2} R_{\alpha \beta}  + \frac{1}{a} \left( 1- \frac{a^2}{b^2} \right) \hat{R}_{\alpha \beta}, \CR
 & & \tilde{R}_{i \alpha} = \left( \frac{a}{b^3} \right)^{\frac{1}{2}} R_{i \alpha},\label{sb-ricci-2}\\
 & & \tilde{R}_{ij} = \frac{a}{b^2} R_{ij} + \frac{1}{b} \left( 1- \frac{a}{b}
 \right)\check{R}_{ij}. \nonumber
 \end{eqnarray}
Subsequently, in our  constructions we shall mainly use formulae (\ref{sb-curvatue-2}) and (\ref{sb-ricci-2}) above.
This consideration will be mainly used for the generalized Hopf fibrations $\Sph^2\to \CP^{2m+1}\to \HP^m$ and $\Sph^7\to \Sph^{15}\to \Sph^8$.

\section{Fateev's 3-dimensional sausage }

 We shall give a detailed presentation of  Fateev's examples of ancient solutions on $\Sph^3$ \cite{Fa96} and discuss their properties. First we start with a parametrization of $\Sph^3$. Write the standard sphere as $|z_1|^2+|z_2|^2=1$ in $\C^2$. Let $z_1=x_1+\sqrt{-1}x_2$ and $z_2=x_3+\sqrt{-1}x_4$. Introduce the parameters $\theta \in [0, \frac{\pi}{2}]$, $\chi_1$ and $\chi_2 \in [0, 2\pi]$ as follows: For any point $(z_1, z_2)$ on $\Sph^3$, there exists a unique $\theta \in [0, \frac{\pi}{2}]$ such that
$$
|z_1|^2=\cos^2 \theta, \quad \quad |z_2|^2=\sin^2 \theta.
$$
For $\theta \in (0, \frac{\pi}{2})$ the level set of $\theta$ is a torus, whereas for $\theta=0$ (or $\frac{\pi}{2}$) is a circle.  The generic level set of $\theta$ can be parametrized by
$\chi_1$ and $\chi_2$ via
\begin{equation}\label{eq-1}
z_1=|z_1|\exp(\sqrt{-1}\chi_1), \quad \quad z_2=|z_2|\exp(\sqrt{-1} \chi_2).\end{equation}
Every point on the sphere except the two circles $$F_1\doteqdot \{ (z_1, 0)\, |\, |z_1|^2=1\}\quad \mbox{ and} \quad  F_2\doteqdot \{(0, z_2)\, |\, |z_2|^2=1\}$$ can be parametrized uniquely by $\theta$ and $\chi_1, \chi_2 \in \R/2\pi \Z$.
With this parametrization in mind, the standard metric of $\Sph^3$ has the form of a doubly-warped product
$$
\de s^2_{\operatorname{stan}}= \de\theta^2 +\cos^2\theta \de\chi_1^2 +\sin ^2\theta \de\chi_2^2.
$$
This can be easily seen from
\begin{eqnarray*}
\de s^2_{\operatorname{stan}}&=&\sum_{i=1}^4 \de x_i^2\\
&=&\frac{1}{2}\left(\de z_1\otimes \de\overline{z}_1+\de\overline{z}_1\otimes \de z_1+\de z_2\otimes \de\overline{z}_2+\de\overline{z}_2\otimes \de z_2\right),\\
\de z_1&=&-\sin \theta e^{\sqrt{-1}\chi_1} \de\theta +\sqrt{-1} \cos \theta e^{\sqrt{-1}\chi_1}\de \chi_1,\\
\de z_2 &=& \cos \theta e^{\sqrt{-1}\chi_2}\de \theta +\sqrt{-1}\sin \theta e^{\sqrt{-1}\chi_2}\de\chi_2.
\end{eqnarray*}
From (\ref{eq-1}) it is also easy to see that $e^{2\sqrt{-1}\chi_1}=\frac{z_1}{\bar{z}_1}$ and $e^{2\sqrt{-1}\chi_2}=\frac{z_2}{\bar{z}_2}$. Differentiating them gives the relations
\begin{equation}\label{eq-2}
\de\chi_1=\frac{x_1\de x_2-x_2\de x_1}{x_1^2+x_2^2},\quad \quad \de\chi_2=\frac{x_3\de x_4-x_4\de x_3}{x_3^2+x_4^2}.
\end{equation}
For convenience in the presentation, we denote  $\phi_1\doteqdot x_1\de x_2-x_2\de x_1, \phi_2\doteqdot x_3\de x_4-x_4\de x_3$.
Fateev's ancient solution is a family of metrics $g(t)$ defined  on $(-\infty, 0)\times \Sph^3$ which solves the Ricci  flow equation
\begin{equation}\label{rf1}
\frac{\partial g}{\partial \tau} =\frac{1}{2}\Ric(g) ,
\end{equation}
where $\tau=-t$, which is now defined on $(0, \infty)$. (We use this nonstandard normalization for convenience.)
 The solution has an ansatz of the following form
\begin{equation}\label{ansatz}
\de s^2_{\nu, k}(\tau)=\frac{1}{w(\tau, \theta)}\left(u(\tau)\de s^2_{\mbox{stan}}+2 d(\tau)(\phi_1^2 +\phi_2^2)+4c(\tau)\phi_1 \phi_2\right) ,
\end{equation}
where
\begin{equation}\label{space}
w(\tau, \theta)=a^2(\tau)-b^2(\tau) (x_1^2+x_2^2-x_3^2 -x_4^2)^2=a^2(\tau)-b^2(\tau)\cos ^2 2\theta
 \end{equation}
 and $a, b, c, d, u$ are functions of $\tau$,
which are given by the formulae
\begin{eqnarray}
a(\tau)&=&\lambda \frac{\sqrt{\cosh^2\xi -k^2\sinh^2 \xi}+1}{\sinh \xi},\nonumber \\
b(\tau)&=&\lambda \frac{\sqrt{\cosh^2\xi -k^2\sinh^2 \xi}-1}{\sinh \xi},\nonumber \\
c(\tau)&=& -\lambda k \tanh \xi, \label{formulae}\\
d(\tau) &=& \lambda \frac{\sqrt{1-k^2\tanh^2 \xi}-\cosh \xi}{\sinh \xi},\nonumber \\
u(\tau)&=& 2\lambda \coth \xi  , \nonumber
\end{eqnarray}
where $\lambda =\frac{\nu}{2(1-k^2)}>0$ and $\nu$ and $k$ are two parameters with $\nu > 0$ and $k^2 < 1$. The new variable
$\xi$  is related to $\tau$ via the equation:
\begin{equation}\label{eq-xitau}
\nu \tau = \xi -\frac{k}{2}\log \left(\frac{1+k\tanh \xi}{1-k \tanh \xi}\right).
\end{equation}

\begin{theorem}[Fateev] The metrics described through the equations (\ref{ansatz}), (\ref{space}) and (\ref{formulae}) are  smooth  ancient solutions
 to the Ricci flow equation (\ref{rf1}).
\end{theorem}

We first check that $\de s_t^2$ is indeed a family of smooth metrics on $\Sph^3$. It is easy to see that $a>b>0$, hence $w>0$.  Note that $w$, which is expressed in terms of $x_1, x_2, x_3, x_4$, can also be viewed as a positive smooth function defined in a small
neighborhood of $\Sph^3$. Note that $u>0, ~d<0$, $u+2d>0$ and $u>2|c|$. Observe also that as symmetric tensors,
\begin{eqnarray*}
4c\,  \phi_1 \phi_2 &\le& 2|c|\left(x^2_3 \de x_2^2+x_1^2 \de x_4^2+x_1^2 \de x_3^2+x_4^2 \de x_2^2+x_3^2 \de x_1^2 +x_2^2 \de x_4^2 +x_4^2 \de x_1^2 +x_2^2 \de x_3^2\right),\\
-2d\, (\phi_1^2 +\phi_2^2) &\le& -2d\left( x_1^2 \de x_2^2 +x_1^2 \de x_1^2 +x_2^2 \de x_2^2 +x_2^2 \de x_1^2\right.\\
&\quad& \left. + x_3^2 \de x_4^2 +x_3^2 \de x_3^2 +x_4^2 \de x_4^2 +x_4^2 \de x_3^2\right).
\end{eqnarray*}
It then follows that
\begin{eqnarray*}
u(\tau)\de s^2_{\mbox{stan}}+2d(\tau)(\phi_1^2 +\phi_2^2)+4c(\tau)\phi_1 \phi_2\ge (u-\max\{ 2|c|, -2d\})\de s^2_{\mbox{stan}}>0.
\end{eqnarray*}
Hence $\de s_t^2$ is a family of smooth positive definite $(2,0)$ symmetric tensors,  even in a small neighborhood of $\Sph^3\subset \R^4$. Moreover $w(\tau, x_1, x_2, x_3, x_4) \de s_t^2$ has the form
\begin{eqnarray*}
&\, &(u+2d\, x_2^2) \de x_1^2 +(u+2d\, x_1^2) \de x_2^2 +(u+2d\, x_4^2) \de x_3^2 +(u+2d\, x_3^2) \de x_4^2  \\
&\,& -4d\, x_1 x_2 \de x_1 \de x_2 -4d\, x_3 x_4 \de x_3 \de x_4 -4c x_2 x_3 \de x_1 \de x_4 -4c x_1 x_4 \de x_2 \de x_3\\
&\,& +4c x_1 x_3 \de x_2 \de x_4 +4c x_2 x_4 \de x_1 \de x_3.
\end{eqnarray*}

To compute the curvature tensor of $g(\tau)=\de s^2_t$ and verify that it is indeed a solution to the Ricci flow equation (\ref{rf1}), it suffices to  work with the coordinates $(\theta, \chi_1, \chi_2)\in (0, \frac{\pi}{2})\times [0, 2\pi)\times [0, 2\pi)$, since this coordinate covers $\Sph^3$ except the two focal sub-manifolds $F_i$ ($i=1, 2$) of codimension $2$. This coordinate becomes singular as $\theta \to 0$ or $\frac{\pi}{2}$. However, the above discussion makes it clear that the metric $g(\tau)$ is nevertheless smooth. With respect to the coordinate $(\theta, \chi_1, \chi_2)$, $g(\tau)$ can be rewritten as follows
$$
g(\tau)=A\de\theta^2 +B \de\chi_1^2 + C \de\chi_2^2 +2D \de\chi_1 d\chi_2,
$$
where $A, B, C, D$ are functions of $\xi$ and $\theta$ only, given by
\begin{eqnarray}
A(\tau, \theta)&=&\frac{u(\tau)}{a^2(\tau)-b^2(\tau)\cos^2 2\theta}\doteqdot \frac{\bA}{w(\tau, \theta)}\nonumber\\
B(\tau, \theta)&=& \cos^2 \theta \frac{u(\tau)+2 d(\tau) \cos^2 \theta}{a^2(\tau)-b^2(\tau)\cos^2 2\theta}\doteqdot \frac{\bB}{w(\tau, \theta)}\nonumber\\
C(\tau, \theta)&=& \sin^2 \theta \frac{u(\tau)+2 d(\tau) \sin^2 \theta}{a^2(\tau)-b^2(\tau)\cos^2 2\theta}\doteqdot \frac{\bC}{w(\tau, \theta)}\label{ABCD}\\
D(\tau, \theta)&=&\frac{2c(\tau) \sin^2 \theta \cos^2 \theta}{a^2(\tau)-b^2(\tau)\cos^2 2\theta}\doteqdot \frac{\bD}{w(\tau, \theta)}.\nonumber
\end{eqnarray}
The right most equalities above define $\bA, \bB, \bC, \bD$.
Next, we need to  demonstrate how to reduce (\ref{rf1}) into a set of ODEs, which then yields the explicit formulae (\ref{formulae}) by solving them. We leave this computational part to the Appendix.

\section{Geometric properties and derived solutions }

Hamilton divided the ancient solutions to Ricci flow on $M \times (-\infty, 0)$ into two types, type-I and type-II, according to the behavior of the curvature. An ancient  solution $g(t)$ is called type-I if there exists a constant $C=C(M)>0$ such that
$$
|\Rm|(x, \tau)\le \frac{C}{\tau}.
$$
Here, as before,  $\tau=-t$. If the above estimate fails, the solution is called type-II.
Recall that Fateev's  family of solutions from Section 3 is a family of two parameters $\nu> 0, -1<k<1$, given by
\begin{equation}
\de s^2_{\nu, k}(\tau)=A(\tau, \theta)\de\theta^2 +B(\tau, \theta) \de\chi_1^2 +C(\tau, \theta)\de\chi_2^2 +2D \de\chi_1 \de\chi_2
\end{equation}
where  $A, B, C, D$ are given in (\ref{ABCD}) with $u(\tau), c(\tau), a(\tau), b(\tau), d(\tau)$ satisfying (\ref{formulae}). The following is easy to check.

\begin{proposition}\label{fateev-ii} The ancient solutions $\de s^2_{\nu, k}(\tau)$ described via (\ref{formulae}) in Section 3 are of type-II.
\end{proposition}
\begin{proof} We follow the notations and computations made in Appendix.
 Direct computation shows that on the focal manifold $F_1$, the Ricci curvature
 $$
 \Ric(e_1, e_1)=g^{11}(\tau, \theta) R_{11}(\tau, \theta) \to 1-k^2, \quad \mbox{as} \quad \tau \to \infty.
 $$
 Here, we set  $e_1=\frac{1}{\sqrt{A}}\frac{\partial}{\partial y_1}$.
\end{proof}

The trivial (Einstein) ancient solution on $\Sph^3$, which is type-I,  can be obtained from the  family $ds^2_{\nu, k}$ by scaling of the space time variables. This is a special case of  the convergence result of Hamilton in three dimensional manifolds with positive Ricci curvature.

\begin{proposition}
As $\nu\to 0$, keeping $k$ fixed,  the metric $\frac{1}{\nu}\de s^2_{\nu, k}(\nu \tau) \to \tau\, \de s^2_{stan}$, the family of Einstein metrics on $\Sph^3$.
\end{proposition}
\begin{proof} This is essentially the well-known theorem of Hamilton, which asserts that as $\tau\to 0$, the rescaled metric converges to constant curvature metric on $\Sph^3$. Indeed, using the fact that
$$
\lim_{\xi \to 0}\frac{f(\xi)}{\xi} =1-k^2
$$
and $\nu^2 \tau=f(\xi)$, which implies that $\lim_{\nu \to 0}\frac{\xi}{\nu^2}=\frac{\tau}{1-k^2}$, we have
\begin{eqnarray*}
\lim_{\nu \to 0} \frac{1}{\nu}\de s^2_{\nu, k}(\nu \tau) &=& \lim_{\nu \to 0}\frac{2(1-k^2)\xi}{\nu^2}\left(\frac{1}{\xi}
 \left(\frac{\bar{u}}{\overline{w}}\de s^2_{\operatorname{stan}}+\frac{2\bar{d}}{\overline{w}}(\phi_1^2+\phi_2^2)
 +\frac{4\bar{c}}
 {\overline{w}}\phi_1\phi_2\right)
 \right)\\
&=& \tau\, \de s^2_{\operatorname{stan}}.
\end{eqnarray*}
Here, $\bar{u}=\frac{u}{\lambda}$, $\bar{c}=\frac{c}{\lambda}$, $\bar{d}=\frac{d}{\lambda}$ and $\overline{w}=\frac{w}{\lambda^2}$.
Note that in the above calculation we fix the parameter $\tau$ and let $\nu\to 0$ (as well as $\xi \to0$).
\end{proof}

For ancient solutions, the so-called $\kappa$ noncollapsing property is important. Recall from \cite{P-entropy} that {\it the metric $g$ (of $M^n$) is called $\kappa$-noncollapsed on the scale $\rho$, if every metric ball $B$ of radius $r<\rho$, which satisfies $|\Rm|(x)\le r^{-2}$ for every $x\in B$, has volume at least $\kappa r^n$.}
In \cite{P-entropy} Perelman proved that every ancient solution arising as a blow-up limit in the singularity of Ricci flow on compact manifolds is $\kappa$-noncollapsed on all scales for some $\kappa>0$. We call an ancient solution {\it collapsed} if there does not exists $\kappa>0$ such that it is $\kappa$-noncollapsed on all scales.
In contrast to Perelman's example we have below:

\begin{proposition} The ancient solutions  $\de s^2_{\nu, k}(\tau)$ are collapsed.
\end{proposition}
\begin{proof} Let $\xi\to \infty$, which is equivalent to $\tau\to \infty$, as $\lim_{\xi \to \infty} \frac{\de f}{\de \xi} \to 1$.
$$
\lim_{\xi \to \infty} \de s^2_{\nu, k}(\tau)=\frac{1}{\nu}\left(\frac{1}{\sin ^2 \theta \cos ^2\theta}\de\theta^2+ \de\chi_1^2 +\de\chi_2^2-2k\, \de\chi_1 \de\chi_2\right)
$$
on $(0, \frac{\pi}{2})\times \mathsf{S}^1 \times \mathsf{S}^1$, which is a collapsed, complete  metric on $\R\times \mathsf{S}^1 \times \mathsf{S}^1$. Hence the family $\de s^2_{\nu, k}(\tau)$ must be  collapsed.
\end{proof}

\begin{corollary} Not every type-II ancient solution is isometric (up to scaling) to the rotationally symmetric example of Perelman.
\end{corollary}
\begin{proof} Since Perelman's example is non-collapsed in all scale, it can not be isometric (after scaling) to $ds^2_{\nu, k}(\tau)$.
\end{proof}

Next, we show that by re-parametrizing and taking the limit of  $\de s^2_{\nu, k}(\tau)$ as $\xi\to \infty$, one can
 obtain the product of the cigar metric on $\R^2$ \cite{form} with $\mathsf{S}^1$. For this, we first describe a special family of $\de s^2_{\nu, k}(\tau)$.

A family of ancient solutions $\de s^2_\nu(\tau)$ can be obtained from $\de s^2_{\nu, k}(\tau)$ simply by letting $k=0$ for which
\begin{eqnarray}
a(\tau)= \lambda \coth \frac{\xi}{2}, &\quad& b(\tau)=\lambda \tanh \frac{\xi}{2},\nonumber\\
c(\tau)=0,&\quad& d(\tau)= -\lambda \tanh \frac{\xi}{2},\\
u(\tau) = 2\lambda \coth \xi, &\quad& \xi=\nu \tau.\nonumber
\end{eqnarray}
Here, $\lambda=\frac{\nu}{2}$ and the coefficients $A, B, C, D$ are given by the expressions
\begin{eqnarray}
A(\tau, \theta)&=& \frac{1}{\nu} \frac{\cosh \xi \sinh \xi}{\left(\cos^2 \theta +\sin ^2 \theta \cosh \xi\right)\left(\sin^2 \theta +\cos ^2 \theta \cosh \xi\right)}, \nonumber\\
B(\tau, \theta)&=&\frac{1}{\nu} \frac{ \cos ^2 \theta \sinh \xi}{\sin ^2 \theta +\cos ^2 \theta \cosh \xi},\nonumber\\
C(\tau, \theta)&=&\frac{1}{\nu} \frac{ \sin ^2 \theta \sinh \xi}{\cos ^2 \theta +\sin ^2 \theta \cosh \xi}, \label{formulaABC}\\
D(\tau, \theta)&=&0. \nonumber
\end{eqnarray}
The metric $\de s^2_\nu(\tau)=A\de\theta^2 +B \de\chi^2 +C \de\chi_2^2$ is a doubly warped product metric, which was first discovered by Fateev in \cite{Fa95}. The regularity of the metric can also be seen from  Proposition \ref{reg1}. In terms of the notation of the last section, the metric has the form
$$
\de s^2_\nu(\tau)=\frac{1}{w(\tau, \theta)}\left((a(\tau)+b(\tau))\de s^2_{\operatorname{stan}}-2b(\tau)(\phi_1^2+\phi_2^2)\right).
$$
The formulae (\ref{formulaABC}) can also be obtained from solving the ODE system:
\begin{eqnarray}
\frac{\de a}{\de \tau} &=& -a(a-b),\label{ode-ab1}\\
\frac{\de b}{\de \tau} &=& b(a-b)\label{ode-ab2}
\end{eqnarray}
which is equivalent to the Ricci flow equation. Note that this system has a simple first integral $a\, b =\operatorname{constant}$.
Letting $e_1=\frac{1}{\sqrt{A}}\frac{\partial}{\partial y_1}$, $e_2=\frac{1}{\sqrt{B}}\frac{\partial }{\partial y_2}$ and $e_3=\frac{1}{\sqrt{C}}\frac{\partial}{\partial y_3}$, the curvature operator of $ds^2_{\nu}(\tau)$ is diagonal with respect to $e_1\wedge e_2$, $e_1\wedge e_3$ and $e_2\wedge e_3$:
\begin{eqnarray*}
&\,&\Rm=\\
&\,&\left(\begin{array}{lll}-\frac{1}{2A}\left(\left(\frac{B'}{B}\right)'+\frac{B'}{2B}
\left(\frac{B'}{B}-\frac{A'}{A}\right)
\right)& \quad\quad  0& \quad \quad 0 \\
0\quad \quad &-\frac{1}{2A}\left(\left(\frac{C'}{C}\right)'
+\frac{C'}{2C}\left(\frac{C'}{C}-\frac{A'}{A}\right)\right)&\quad \quad 0 \\
0 \quad\quad & \quad \quad 0& -\frac{1}{4A}\frac{B'}{B}\frac{C'}{C}\end{array}\right)\\
&\quad &\quad =(a-b)\left(\begin{array}{ccc}  -\frac{a-b}{a+b}+2\frac{a-b\cos 2\theta}{a+b\cos 2\theta}& \quad\quad  0& \quad \quad 0 \\
0\quad \quad & -\frac{a-b}{a+b}+2\frac{a+b\cos 2\theta}{a-b\cos 2\theta}&\quad \quad 0 \\
0 \quad\quad & \quad \quad 0& \quad \quad\frac{a-b}{a+b} \end{array}\right).
\end{eqnarray*}
It is worthwhile to mention that this solution has pinched sectional curvature with pinching constant that tends to zero as $\tau\to \infty$. Since its curvature operator has three different eigenvalues generically, it can not be  rotationally symmetric. Clearly this metric is not homogenous.

An interesting feature of this ancient solution is that one can obtain Hamilton's cigar solution by taking a suitable limit of the metric as $\tau\to \infty$. Recall that Hamilton's cigar is a metric on $\R^2$ which is gradient steady soliton. Under the cylindrical coordinate, it can be expressed as
$$
\de s^2_{\operatorname{cigar}}=\frac{1}{\nu} \frac{ \de x^2+ \de y^2}{1+e^{2y}} ,
$$
where $(x, y)\in \mathsf{S}^1\times \R$. (As before we identify $\R/2\pi\Z$ with $\mathsf{S}^1$.)

Next, introduce a new variable $\tilde{y}$ such that $\tanh \tilde{y}=\cos ^2 \theta -\sin^2 \theta$. It is easy to see that $\tilde {y}\in\R$ and
$$
\de s^2_{\nu}(\tau)=\frac{\sinh \xi}{\nu}\left(\frac{\cosh \xi  \de \tilde{y}^2}{(e^{2\tilde{y}}+e^{-2\tilde{y}})\cosh \xi +1+\cosh^2 \xi} +\frac{\de\chi_1^2}{e^{-2\tilde{y}}+\cosh \xi}+\frac{\de\chi_2^2}{e^{2\tilde{y}}+\cosh\xi} \right).
$$
Letting at this point $\tilde{y}= y+\frac{\xi}{2}$ and taking $\xi\to \infty$, we have finally
$$
\de s^2_{\nu}(\tau) \to \frac{1}{\nu}\left( \frac{\de y^2+\de\chi_2^2}{1+2e^{2y}}+\de\chi_1^2\right).
$$
(Another simple translation takes the above into the standard form of Hamilton's cigar metric.) Hence, we have the following:
\begin{proposition}
After the change of variables described above, $\de s^2_{\nu}(\tau)$ converges to the product of Hamilton's  cigar with  $\mathsf{S}^1$ as $\tau\to \infty$.
\end{proposition}

It is a little surprising, but in fact one can obtain a family of type-I ancient solutions from $\de s^2_{\nu, k}(\tau)$ by a suitable limiting process. By Proposition \ref{fateev-ii}, to obtain a type-I solution one has to let $k\to 1$. Indeed, let $k\to 1$ and $\nu\to 0$, but in the manner that
$$
\frac{2\nu}{1-k^2}=\Omega
$$
is a fixed number. Noting that $\lambda=\frac{\nu}{2(1-k^2)}=\frac{\Omega}{4}$,
the relation $\nu\tau =f(\xi)$ becomes
\begin{eqnarray*}
\Omega \tau &=&2\lim_{k\to 1}\frac{\xi -\frac{k}{2}\log \left(\frac{1+k\tanh \xi}{1-k \tanh \xi}\right)}{1-k^2}\\
&=& 2\lim_{k\to 1}\frac{1}{1+k} \left( \frac{1}{2}\log \left(\frac{1+k\tanh \xi}{1-k \tanh \xi}\right)+\frac{k}{2}\frac{2\tanh \xi}{1-k^2 \tanh^2 \xi}\right)\\
&=& \xi +\frac{2\sinh 2\xi}{2}.
\end{eqnarray*}
It is easy to check in this case that
\begin{eqnarray*}
a(\tau)=\frac{\Omega}{2\sinh \xi}, &\quad& b(\tau)=0,\\
c(\tau)=d(\tau)=-\frac{\Omega}{4}\tanh \xi, &\quad & u(\tau)=\frac{\Omega}{2}\coth \xi.
\end{eqnarray*}
Then, the limit metric has the form
\begin{eqnarray*}
\de s^2_{\Omega}(\tau)&=&\frac{\sinh 2\xi}{\Omega}\left(\de\theta^2 +\cos^2\theta(1-\tanh^2\xi \cos^2\theta)\de\chi_1^2 +\sin^2\theta(1-\tanh^2\xi \sin^2\theta)\de\chi_2^2 \right.
\\&\quad& \left. -2\sin^2\theta \cos^2 \theta \tanh^2 \xi \de\chi_1\de\chi_2\right)\\
&=& \frac{\sinh 2\xi}{\Omega}\left(\de s^2_{\operatorname{stan}}-\tanh ^2 \xi (\phi_1^2 +\phi_2^2 +2\phi_1\phi_2)\right).
\end{eqnarray*}

\begin{proposition} \label{fateev-typei} The family of metrics $\de s^2_{\Omega}(\tau)$ are type-I collapsed ancient solutions.
\end{proposition}
\begin{proof}
Now introduce the following change of variables:
$$
\Theta =2\theta,\quad \Phi=\frac{\chi_1+\chi_2}{2}, \quad \Psi =\frac{\chi_1-\chi_2}{2}.
$$
Introduce the $1$-forms
\begin{eqnarray*}
\psi_1&=&\sin \Phi \, \de\Theta -\sin \Theta \cos \Phi \, \de \Psi , \\
\psi_2&=& -\cos \Phi\,  \de \Theta -\sin \Theta \sin \Phi\,  \de \Psi , \\
\psi_3&=& -\de\Phi-\cos \Theta\,  \de\Psi.
\end{eqnarray*}
Direct calculation shows that
$$
\de s^2_{\Omega}=\frac{\sinh 2\xi}{\Omega} \left(\psi_1^2+\psi_2^2\right)+\frac{2\tanh \xi}{\Omega} \psi_3^2.
$$
Viewing $\Sph^3$ as the total space of the Hopf fibration over $\CP^1$,
it is easy to check that
$$
\psi_1^2+\psi_2^2 =\de\Theta^2 +\sin^2 \Theta \de\Psi^2
$$
corresponds to the metric on the base manifold $\CP^1$. Hence, $\{\psi_1, \psi_2\}$ form a moving frame of the base manifold $\CP^1$. Also $\de\psi_3=-\psi_1\wedge \psi_2$, which is the $-1$ multiple of the K\"ahler form. Hence $\psi_3$ can be viewed
as a connection $1$-form on the total
space; in fact,  this example fits into the generalization considered in the next section. The rest of the proof is a special case of Theorem \ref{main1-u1}. \end{proof}

The type-I example above  was also recently studied in  Theorem 2.1 of \cite{CS}.

\section{Type-I ancient solutions on a $\mathsf{U}(1)$-bundle over a K\"ahler-Einstein manifold of positive scalar curvature}
In this section, we shall construct examples generalizing the metrics $\de s^2_{\Omega}$ in Proposition \ref{fateev-typei}.
 First, recall the computations in Section 2.2 on the general connection metric on a principle bundle. If we consider a  $\mathsf{U}(1)$-bundle $P$, the Lie group/algebra is trivial and the Riemannian curvature tensor and Ricci curvature of the variation metric $\widetilde{g}_{a, b}$   are simply given by
 \beqn
 \widetilde{R}_{i0 j 0} &=& \frac{a}{4b^2} F_{ik} F_{jk}, \CR
 \widetilde{R}_{ijk 0} &=& - \frac{1}{2} \left( \frac{a}{b^3} \right)^{\frac{1}{2}} F_{ij,k}, \CR
 \widetilde{R}_{ijkl} &=& \frac{1}{b} R_{ijkl} - \frac{a}{4b^2} \left( 2 F_{ij} F_{kl}
 + F_{ik}F_{jl} - F_{il} F_{jk} \right),
 \eeqn
 and
 \beqn
 \widetilde{R}_{00} &=& \frac{a}{4b^2} F_{ij} F_{ij}, \CR
 \widetilde{R}_{i \alpha} &=& \frac{1}{2} \left( \frac{a}{b^3} \right)^{\frac{1}{2}} F_{ij,j}, \CR
 \widetilde{R}_{ij} &=& \frac{1}{b} R_{ij} - \frac{a}{2b^2} F_{ik} F_{jk}.
 \eeqn
We further restrict ourselves to the case that  $(M^{2m}, J, g)$ is a compact K\"ahler-Einstein manifold such that $\Ric(g) = p \, g$ for some $p>0$  and $P$ a principle $\mathsf{U}(1)$-bundle with a connection $1$-form $\sqrt{-1} \theta$ such
 that its curvature satisfies
 \[
 \de \theta = q  \omega
 \]
 for some $q\ne 0$,
 where $\omega$ is the K\"ahler form of $(M, g)$. If we normalize so that $\omega$ is an integral class, then $p$ and $q$ can only take rational values (one can even further normalize so that $p=1$). The typical examples include the $\mathsf{U}(1)$-bundle over $\CP^{m}$.

 \begin{theorem} \label{main1-u1} Let $(M, g)$ be a K\"ahler-Einstein manifold with positive Chern class and
   let $P$ be a $\mathsf{U}(1)$ principle bundle over $M$ with a connection $1$-form $\theta$ such that its curvature is a nonzero multiple of the K\"ahler form. There exist positive functions $a_\Lambda(\tau)$ and $b_\Lambda(\tau)$ on $(0, \infty)$ (depending on a parameter $\Lambda$) such that $\widetilde{g}_{a, b} =a\, \theta( \cdot)\otimes \theta(\cdot) + b\,  \pi^* g $ is an ancient solution to Ricci flow on the total space $P^n$ ($n=2m+1$). Moreover, the solution is of type-I and collapsed. It has    positive curvature operator when $(M, g)$ is $(\CP^{m}, c\, g_{\operatorname{FS}})$, where $g_{\operatorname{FS}}$ is the Fubini-Study metric and  $c>0$ is a constant.
 \end{theorem}

\begin{remark} Since $M$ is algebraic, for any $q$ such that $q\omega \in H^{1, 1}(M, \C)\cap H^2(M, \Z)$ by Lefschetz theorem there always  exists a $\mathsf{U}(1)$ bundle and a connection $\theta$ such that its curvature form is $q\omega$. Under the assumption of the theorem the existence of Einstein metrics on $P$ was known first in \cite{Kob}.
\end{remark}

Since the Einstein metrics on spheres are non-collapsed ancient solutions to the Ricci flow, we have the following immediate consequence.

 \begin{corollary} In the classification result of type-I ancient solutions in \cite{Ni-anc},
the non-collapsed condition can not be removed.
\end{corollary}

 \noindent {\sl Proof} (of Theorem \ref{main1-u1}). For simplicity we write $\widetilde{g}$ for $\widetilde{g}_{a, b}$.
 First, observe that  $F_{ij} = q \omega_{ij}$, hence
$F_{ij,k} = 0$ and
 $F_{ik}F_{jk} = q^2  \delta_{ij}$.
 The Riemannian curvature tensor can be simplified:
 \beqn
 \widetilde{R}_{i0 j 0} &=& \frac{q^2}{4} \frac{a}{b^2} \delta_{ij}, \CR
 \widetilde{R}_{ijk 0} &=& 0, \CR
 \widetilde{R}_{ijkl} &=& \frac{1}{b} R_{ijkl} - \frac{q^2}{4} \frac{a}{b^2} \left( 2 \omega_{ij} \omega_{kl}
 + \omega_{ik}\omega_{jl} - \omega_{il} \omega_{jk} \right).
 \eeqn
Hence, the Ricci curvature is given by
 \beqn
 \widetilde{R}_{00} &=& \frac{mq^2}{2} \frac{a}{b^2},  \CR
 \widetilde{R}_{i \alpha} &=& 0, \CR
 \widetilde{R}_{ij} &=& \left( \frac{p}{b}  - \frac{q^2}{2} \frac{a}{b^2}
 \right) \delta_{ij}.
 \eeqn
The Ricci tensor of $\widetilde{g}$ is of the form
 $
 \Ric(\widetilde{g}) = \frac{m q^2}{2}  \frac{a^2}{b^2} \theta \otimes \theta +
 \left( p - \frac{q^2}{2}
 \frac{a}{b} \right) g
 $
 and  the Ricci flow equation
 $
 \frac{\partial \widetilde{g}}{\partial \tau} = 2 \Ric (\widetilde{g})
 $
 (with $\tau=t_0-t$) is reduced to the following ODE system
 \begin{eqnarray} \label{typei-ode1}
 \frac{\de a}{\de \tau} &=&  m q^2 \frac{a^2}{b^2}, \\
 \frac{\de b}{\de \tau} &=& 2 p - q^2 \frac{a}{b}. \label{typei-ode2}
 \end{eqnarray}
 To solve this ODE, observe that
 there is a first integral of this system,
 \[
 \left( \frac{2p}{(m+1)q^2} - \frac{a}{b}  \right) a^{-\frac{m+1}{m}} = \Lambda^{\frac{m+1}{m}}
 \]
 where $\Lambda$ is a constant. To see this,  let $y=\frac{a}{b}$ and use (\ref{typei-ode1})  and (\ref{typei-ode2}) to obtain
 the equation \begin{equation}\label{typei-ode3}
 \frac{\de y}{\de \tau}=y\left((m+1)q^2y-2p\right)\frac{1}{b}.
 \end{equation}
The first integral is obtained by solving the equation that arises by dividing (\ref{typei-ode1}) and (\ref{typei-ode3}).

{\it Case 1}:  $\Lambda=0$. Then,
 \[
 \frac{a}{b} = \frac{2p}{(m+1)q^2} .
 \]
 By (\ref{typei-ode1}) and (\ref{typei-ode2}), if we require that $\lim_{\tau\to 0} a(\tau)=\lim_{\tau\to 0} b(\tau) =0$,
 \beqn
 a &=&  \frac{4m p^2}{(m+1)^2 q^2}  \tau, \CR
 b &=& \frac{2 m p}{m+1} \tau .
 \eeqn
 Hence, the metric $\widetilde{g}(\tau) = 2 \tau g_{e}$, where
 \[
 g_e = \frac{2m p^2}{(m+1)^2 q^2} \theta \otimes \theta + \frac{m p}{m+1} g
 \]
 is an Einstein metric such that
$
 \Ric(g_e) = g_e
$.  So, we obtain a trivial solution.

{\it Case 2}:  $\Lambda >0$. Then,
 \beq \label{e1}
 \frac{a}{b} = \frac{2p}{(m+1)q^2} - (\Lambda a)^{\frac{m+1}{m}}
 \eeq
 and (\ref{typei-ode1}) and (\ref{typei-ode2}) become
 \begin{eqnarray} \label{e2}
 & & \frac{\de a}{\de \tau} =  m q^2 \left( \frac{2p}{(m+1)q^2} - (\Lambda a)^{\frac{m+1}{m}}
 \right)^2, \\
 & & \frac{\de b}{\de \tau} = \frac{2 m p}{m+1} +  q^2 (\Lambda a)^{\frac{m+1}{m}}.\label{e5}
 \end{eqnarray}
 It is relatively easy to prove the long time existence of the solutions satisfying $\lim_{\tau \to 0} a(\tau)=\lim_{\tau \to 0} b(\tau)=0$. Due to (\ref{e1}), one only has to solve (\ref{e2}). Since $a$ and $b$ are increasing functions of $\tau$, (\ref{e1}) implies that  $a$ stays bounded from above  by a fixed number. Then, we conclude that (\ref{e2}) has global solution on $(0, \infty)$ by, say Theorem 7 of \cite{H}. In fact, the solution is also unique by the same result.

The case $\Lambda^{\frac{m+1}{m}}<0$ is not interesting for our consideration, since the solution to (\ref{e2}) will have  finite time blow-up.
 Next,  we check that the solution $\widetilde{g}_{a, b}$ is of type-I and collapsed.  Since
 \beq \label{e3}
 \frac{\de b}{\de \tau} \geq \frac{2 m p}{m+1},
 \eeq
we know that  $b \rightarrow \infty$, as $\tau \rightarrow \infty$. It follows from $(\ref{e1})$ that as $\tau\to \infty$,
  \beq \label{e4}
  a \rightarrow \frac{1}{\Lambda}\left( \frac{2p}{(m+1)q^2} \right)^{\frac{m}{m+1}}.
  \eeq
By the formula on the curvature, it is not hard to see that there exists $C$, depending only on  $p, q, m$, such that
  \[
  |\widetilde{\Rm}| \le \frac{C}{b}.
  \]
  Since $|\frac{\de b}{\de \tau}|$ is bounded from above, which implies that $\frac{b}{\tau}$ is bounded from above,
  we have that
  \[
   |\widetilde{\Rm}| \tau \le C'.
  \]
  This shows that the solution $\widetilde{g}_{a, b}$ is type-I.

   Finally, by
  $(\ref{e4})$,  the fibre $\mathsf{S}^1 $ of the bundle $P$ has length bounded  from above as $\tau
  \rightarrow \infty$, while the curvature goes to zero as $\tau \to \infty$ in the rate of $\frac{1}{\tau}$.  Hence, the metric $\widetilde{g}_{a, b}$ must be    collapsed.

  The last claim on the positivity of the curvature operator follows from the proposition below, (\ref{e1}), $a>0$, $\Lambda>0$ and the fact that $\frac{2}{m+1}<\frac{4}{2m+1}$.\qed

  \begin{proposition}
 Let $g$ be a multiple of the Fubini-Study metric on $\mathbb{C}P^m$ with $\Ric(g) = p g$. Let  $\omega$ be its K\"ahler form. Assume that  $P$ is a $\mathsf{U}(1)$-bundle with connection $\theta$ such that
$
\de \theta = q \omega,
$
for some $q$.
Then the curvature operator of the metric $\widetilde{g}_{a, b} = a \theta\otimes \theta + b g$
 on $P$ is positive if and only if
 \[
 \frac{a}{b} < \frac{4}{2m+1} \frac{p}{q^2}.
 \]
 \end{proposition}
\begin{proof}
 Let $\{e_1, \cdots, e_{2m} \}$ be an orthonormal tangent
 vector of $(\mathbb{C}P^m,g)$ and $e_{m+k} = Je_k$, $1 \leq k \leq m$. Denote $\omega_{ij} = \omega(e_i,
e_j)$. By O'Neill's formula, the Riemannian curvature of $g$ is given by
 \[
 R_{ijkl} = \frac{p}{2m+2} \left( \delta_{ik} \delta_{jl} -\delta_{il} \delta_{jk} + 2
 \omega_{ij} \omega_{kl} + \omega_{ik} \omega_{jl} - \omega_{il} \omega_{jk}
\right).
 \]
 Hence,
 \beqn
 \widetilde{R}_{ijkl} &=&  \frac{p}{2m+2} \frac{1}{b} \left( \delta_{ik} \delta_{jl} -\delta_{il} \delta_{jk}
\right)  \CR
 & & + \left(  \frac{p}{2m+2} \frac{1}{b} - \frac{q^2}{4} \frac{a}{b^2} \right) \left( 2 \omega_{ij} \omega_{kl}
 + \omega_{ik}\omega_{jl} - \omega_{il} \omega_{jk} \right) \CR
 &=& \frac{p}{2m+2} \frac{1}{b} \left( \delta_{ik} \delta_{jl} -\delta_{il} \delta_{jk}
\right)  \CR
 & & + \frac{p}{2m+2} \frac{1}{b} \left(  1 - \frac{m+1}{2} \frac{q^2}{p} \frac{a}{b} \right) \left( 2 \omega_{ij} \omega_{kl}
 + \omega_{ik}\omega_{jl} - \omega_{il} \omega_{jk} \right).
 \eeqn
 As an algebraic curvature operator (namely a symmetric tensor of $\wedge^2(TP)$), the eigenvalues of $2 \omega_{ij} \omega_{kl} + \omega_{ik}\omega_{jl} - \omega_{il} \omega_{jk}$  are given by the following
table:
 \[
 \begin{tabular}{|c|c|c|} \hline
 eigenvalues & multiplicities  &   eigenvectors      \CR \hline
 $2m+1$  &  $1$   & $\sum_{i=1}^m e_i \wedge e_{m+i}$   \CR \hline
 $1$  &  $m^2-1$   & $\begin{matrix} e_i \wedge e_{m+i} - \frac{1}{m} \left(\sum_{i=1}^m e_i \wedge e_{m+i}\right), \CR e_i \wedge e_j + e_{m+i} \wedge e_{m+j},  \CR
 e_i \wedge  e_{m+j} - e_{m+i} \wedge e_j,  \end{matrix}$  \CR \hline
 $-1$  &  $m^2-m$   & $\begin{matrix} e_i \wedge e_j - e_{m+i} \wedge e_{m+j},  \CR
 e_i \wedge  e_{m+j} + e_{m+i} \wedge e_j,  \end{matrix}$    \CR \hline
 \end{tabular}
 \]
 where $1 \leq i \neq j \leq m$. It follows that the eigenvalues of the curvature operator $\widetilde{\Rm}$ are given by
 \[
 \frac{p}{b} \left( 1 - \frac{2m+1}{4} \frac{q^2}{p} \frac{a}{b} \right), ~~~\frac{p}{m+1} \frac{1}{b} \left( 1 -  \frac{m+1}{4} \frac{q^2}{p} \frac{a}{b}
 \right), ~~~ \frac{q^2}{4} \frac{a}{b^2}
 \]
 with multiplicities $1$, $m^2-1$ and $m^2+m$, respectively.
 Clearly, all of these eigenvalues are positive if and only if
 \[
 \frac{a}{b} < \frac{4}{2m+1} \frac{p}{q^2}.
 \]
 \end{proof}

 The convergence theorem of B\"ohm-Wilking \cite{BW} has the following manifestation.
 \begin{proposition} Let $\de s^2_{\Lambda}(\tau)=a_{\Lambda}(\tau)\theta\otimes \theta +b_{\Lambda}(\tau) g$ be the metric as in Theorem \ref{main1-u1}. Then, $\frac{1}{\Lambda} \de s^2_{\Lambda}(\Lambda \tau) \to 2\tau\,  g_{e}$, where $g_{e}$ is the trivial ancient solution (Einstein metric) in the proof of Theorem \ref{main1-u1},  as $\Lambda\to 0$.
 \end{proposition}
 \begin{proof} Observe that by the uniqueness of the ODE,  $\frac{1}{\Lambda}a(\Lambda \tau)$ and $\frac{1}{\Lambda} b(\Lambda \tau)$ are equal to $a_{\Lambda^2}(\tau)$ and $b_{\Lambda^2}(\tau)$. Hence, the result follows from the smooth dependence of the solutions of the ODE system (\ref{e2}), (\ref{e5}) on the parameter $\Lambda$. \end{proof}

\begin{remark} When $m=1$, the ODE system can be solved explicitly. First $(\ref{e1})$ and $(\ref{e2})$ are reduced to
$$
 \frac{\de a}{\de \tau} =   q^2 \left( \frac{p}{q^2} - (\Lambda
 a)^2  \right)^2
 \quad \quad \mbox{
and } \quad
 \frac{a}{b} = \frac{p}{q^2} - (\Lambda a)^2.
 $$
 Let $\Lambda^2 =  \frac{p}{q^2} \nu^2$. Then,
$$
 \frac{\de a}{\de \tau} =   \frac{p^2}{q^2} \left( 1 -
 \nu^2 a^2  \right)^2
 \quad \quad \mbox{
 and }\quad
 \frac{a}{b} = \frac{p}{q^2} ( 1 -
 \nu^2 a^2  ).
 $$
 Set $a = \frac{1}{\nu} \tanh \xi$. Then, $b = \frac{1}{\nu} \frac{q^2}{p} \sinh
 \xi \cosh \xi$ and $\xi$ is the function of $\tau$ determined by
 \[
 \xi + \frac{1}{2} \sinh (2 \xi) = \frac{p^2}{q^2} 2 \nu \tau.
 \]
\end{remark}

Theorem \ref{main1-u1}, in particular, can be applied to the Hopf fibration $\mathsf{S}^1 \to \Sph^{2m+1}\to \CP^m$. This corresponds to $p=2(m+1)$, $q=-1$. Then, $\widetilde{g}_{a, b}(\tau)$ is an ancient solution on $\Sph^{2m+1}$, which is type-I, has positive curvature operator and it is collapsed.

\section{Ancient solutions on principle $\mathsf{SU}(2)$-bundles over quaternion-K\"ahler manifolds with positive scalar curvature}

In this section we shall generalize the construction of the previous section by allowing the fiber to be $\mathsf{SU}(2)$. However, to make it work, we have to restrict the base manifolds to the quaternion-K\"ahler ones. We first recall some basic properties of quaternion-K\"ahler
 manifolds for completeness. These properties were proved by
 Berger  \cite{berger} (see also \cite{ishi} by Ishihara and Chapter 14 of \cite{Besse}). A quaternion-K\"ahler manifold $(M,g)$ is a Riemannian manifold with a rank
 $3$ vector bundle $V \subset \operatorname{End}(TM)$ satisfying:
 \begin{enumerate}
 \item[$(a)$] In any coordinate neighborhood $U$ of $M$, there exists
 a local basis $\{I,J,K\}$ of $V$ such that
  \beqn
  & & I^2=J^2=K^2 = - \id, \quad \quad
   IJ=-JI = K, \CR
  & & JK = -KJ =I, \quad \quad \quad
   KI=-IK =J
  \eeqn
 and
 \[
 \langle I(X),I(Y)\rangle =\langle J(X),J(Y)\rangle =\langle K(X),K(Y)\rangle =\langle X,Y\rangle
 \]
 for all $X,Y \in TM$.
 \item[$(b)$] If $\phi \in \Gamma(V)$, then $\nabla_X \phi \in
 \Gamma(V)$ for all $X \in TM$.
 \end{enumerate}

 It follows from $(a)$ that $\dim M = 4m$. The condition $(b)$ implies that
 there are local $1$-forms $\sigma_1,\sigma_1,\sigma_3$ such that
  \[
  (\nabla I, \nabla J, \nabla K ) = (I, J,K) \cdot \left(\begin{matrix} 0 & -\sigma_3 & \sigma_2 \cr \sigma_3 & 0 & -\sigma_1 \cr -\sigma_2 & \sigma_1 & 0
 \end{matrix} \right).
  \]
 Let $\omega_1, \omega_2, \omega_3$ be three $2$-forms defined by
  \beqn
  \omega_1(\cdot, \cdot) \doteqdot \langle \cdot, I (\cdot)\rangle,  \quad \quad
  \omega_2(\cdot, \cdot) \doteqdot \langle \cdot, J (\cdot)\rangle, \quad \quad
  \omega_3(\cdot, \cdot) \doteqdot \langle \cdot, K (\cdot)\rangle
  \eeqn
 and let  $\Omega$ be a $4$-form defined by
 \[
 \Omega = \omega_1 \wedge \omega_1 +  \omega_2 \wedge \omega_2 + \omega_3 \wedge
 \omega_3 .
 \]
 The condition $(b)$ is equivalent to $\Omega$ being parallel, that is
 $
 \nabla_X \Omega = 0
 $
 for any $X \in TM$.

 The curvature properties of quaternion-K\"ahler manifold can be summarized
 in the following theorem. See \cite{ishi} or Chapter 14 of \cite{Besse} for a proof.

  \begin{theorem}
  If $(M^{4m},g)$ is a quaternion-K\"ahler manifold and $m \geq 2$,
  then $(M^{4m},g)$ is Einstein, that is, there is a constant
  $p$ such that
 $
  Ric(g) = p g.
  $
  Moreover,
  $$
  \de \sigma_1 + \sigma_2 \wedge \sigma_3 = \frac{p}{m+2} \omega_1, \quad \quad
  \de \sigma_2 + \sigma_3 \wedge \sigma_1 = \frac{p}{m+2} \omega_2, \quad \quad
  \de \sigma_3 + \sigma_1 \wedge \sigma_2 = \frac{p}{m+2} \omega_3.
  $$
  \end{theorem}

  Let $P_0$ be the $\mathsf{SU}(2)$-principle bundle associated with the rank $3$ bundle $V$. In the following we assume that $p>0$. One can
 identify the Lie algebra $\mathfrak{su}(2)$ with $\mathbb{R} \{{\bf i}, {\bf j},{\bf k} \}$, where
 ${\bf i,j,k}$ are the canonical quaternionic numbers. Then,
 \[
 A = \frac{1}{2} ( \sigma_1 {\bf i} + \sigma_2 {\bf j} + \sigma_3 {\bf k} )
 \]
 defines a connection on $P_0$, and by the above theorem, the curvature $F_A$ of
 $A$ is given by
 \[
 F_A = \frac{p}{2(m+2)} (\omega_1  {\bf i} + \omega_2 {\bf j} +  \omega_3 {\bf k} ).
 \]

  Now we use the notation and computations of Section 2.2 to vary the connection metric on a principle bundle $\mathsf{SU}(2)$-principle bundle $P$ over a quaternion-K\"ahler manifold
 $M^{4m}$ with a connection $A$ so that its curvature $F_A$ satisfies
 \begin{equation}\label{qk-p-curv}
 F_A = q (\omega_1  {\bf i} + \omega_2 {\bf j} +  \omega_3 {\bf k} ).
 \end{equation}
Since  the structure constant $C^\alpha_{\beta \gamma}$ is totally skew-symmetric and
 $C^1_{23}=2$, we have that
  \[
  \sum_{\gamma, \sigma} C^{\gamma}_{\alpha \sigma} C_{\beta
  \sigma}^{\gamma} = 8 \delta_{\alpha \beta}.
  \]
 Using formulae from Section 2.2,  the  curvature $F_{ij}^{\alpha}$ is computed as
 \beqn
 F_{ij, k}^{\alpha} &=& 0, \CR
 \sum_{i,j} F_{ij}^{\alpha} F_{ij}^{\beta} &=& 4 m q^2 \delta_{\alpha \beta}, \CR
\sum_{k, \alpha} F_{ik}^{\alpha}
  F_{jk}^{\alpha} &=& 3 q^2 \delta_{ij}.
 \eeqn
 Therefore by (\ref{principle-ricci}) the Ricci curvature of the connection metric $\widetilde{g}_{a, b}$ (which is defined in Section 2.2) on the total space of $\mathsf{SU}(2)$-principle bundle simplifies to
  \beqn
  \widetilde{R}_{\alpha \beta} &=& \frac{1}{a} \left( 2+ mq^2 \frac{a^2}{b^2} \right) \delta_{\alpha \beta} ,\CR
  \widetilde{R}_{i \alpha} &=& 0, \CR
  \widetilde{R}_{ij} &=& \frac{1}{b} \left( p- \frac{3q^2}{2} \frac{a}{b} \right) \delta_{ij}.
  \eeqn
It follows that the Ricci flow equation
 $
 \frac{\partial \widetilde{g}_{a, b}(\tau)}{\partial \tau} = 2 \Ric (\widetilde{g}_{a, b}(\tau)),
 $ where $\tau=t_0-t$, is reduced to the  ODE system
 \begin{eqnarray}\label{qk-ode1}
 \frac{\de a}{\de \tau} &=&  4+ 2mq^2 \frac{a^2}{b^2}, \\
 \frac{\de b}{\de \tau} &=& 2 p - 3q^2 \frac{a}{b}. \label{qk-ode2}
 \end{eqnarray}
 To solve this system, let $y= \frac{a}{b}$. Also, let $y' = \frac{\de y}{\de \tau}$ (and similarly for $b'$). Then,
 \begin{equation}\label{ode-y}
y' = \frac{1}{b} \left(4 -2p y +(2m+3)q^2 y^2 \right)
 \end{equation}
 and
 \[
\frac{y'}{b'} = \frac{1}{b} \frac{4 -2p y +(2m+3)q^2 y^2}{2 p - 3q^2 y}.
 \]
 Separating $y$ and $b$, one gets
 \[
 \frac{2 p - 3q^2 y}{4 -2p y +(2m+3)q^2 y^2} y' = \frac{b'}{b}.
 \]
 Notice that, if $p>\sqrt{4(2m+3)q^2}$,
 \[
 \frac{2 p - 3q^2 y}{4 -2p y +(2m+3)q^2 y^2} = \frac{\alpha_1}{y-r_1} -
\frac{\alpha_2}{y-r_2} ,
 \]
 where
 \beqn
 r_1 &=& \frac{p+ \sqrt{p^2-4(2m+3)q^2}}{(2m+3)q^2}, \quad \quad \quad \quad \quad \quad
 r_2 = \frac{p- \sqrt{p^2-4(2m+3)q^2}}{(2m+3)q^2}, \CR
 \alpha_1 &=& \frac{(4m+3)p -3 \sqrt{p^2-4(2m+3)q^2}}{2(2m+3)\sqrt{p^2-4(2m+3)q^2}}, \quad \quad
 \alpha_2 = \frac{(4m+3)p +3 \sqrt{p^2-4(2m+3)q^2}}{2(2m+3)\sqrt{p^2-4(2m+3)q^2}}.
 \eeqn
It follows that the above ODE system of $a$ and $b$ has a complete integral
\begin{equation}\label{qk-int}
 \left|  \frac{a}{b}-r_1 \right|^{\alpha_1}
 \left| \frac{a}{b} -r_2\right|^{-\alpha_2} b^{-1} =
 \Lambda,
 \end{equation}
 where $\Lambda\ge 0$ is a constant.

In the special case that the $\mathsf{SU}(2)$-principle bundle is just the two-fold lift of the principle bundle associated with $V$ (namely $P=P_0$) and $A=\frac{1}{2} ( \sigma_1 {\bf i} + \sigma_2 {\bf j} + \sigma_3 {\bf k} )$, which implies that
$
 q= \frac{p}{2(m+2)},
$
 the constants $r_1, r_2, \alpha_1, \alpha_2$ take the simple form:
 \beqn
 r_1 &=& \frac{4(m+2)}{p},  \quad \quad
 r_2 = \frac{4(m+2)}{(2m+3)p} , \CR
 \alpha_1 &=& \frac{2m+1}{2(m+1)}, \quad \quad
 \alpha_2 = \frac{4m^2+14m+9}{2(m+1)(2m+3)}.
 \eeqn

 When  $\Lambda =0$ we obtain  two Einstein metrics on
$P$ (cf.  \cite{J}),
which we denote by $\widetilde{g}_{e_1}$ and $\widetilde{g}_{e_2}$, corresponding to the `slope' $y=\frac{a}{b}$ being $r_1$ or $r_2$. Notice that $r_1>r_2$ and $\alpha_2>\alpha_1$, if $p>\sqrt{4(2m+3)q^2}$.

\begin{theorem} \label{qk-main1} Assume that $(M, g)$ is a quaternion-K\"ahler manifold with Einstein constant $p>0$. Let $P$ be the associated $\mathsf{SU}(2)$-principle bundle with connection $A$ satisfying (\ref{qk-p-curv}). Assume that $p>2\sqrt{(2m+3)q^2}$.  Then, there exists a type-I ancient solution $\widetilde{g}_{a, b}$ to Ricci flow on the total space $P$ with $r_2<y(\tau)<r_1$, which flows (after re-normalization) the Einstein metric $\widetilde{g}_{e_2}$ (which corresponds to $y=r_2$) into the Einstein metric $\widetilde{g}_{e_1}$ (which corresponds to $y=r_1$) as $t$ increases from $-\infty$ to some $t_0$.

There also exists a type-I ancient solution $\widetilde{g}_{\tilde{a}, \tilde{b}}$ to Ricci flow on the total space $P$ with $r_2>y(\tau)>0$ which flows the Einstein metric $\widetilde{g}_{e_2}$ from $t=-\infty$ into a singularity at time $t_0$ when it collapses the $\mathsf{SU}(2)$ fiber.
\end{theorem}
\begin{proof}

The proof of the first part is the same as that of Theorem \ref{sb-main1} next. We refer readers to the next section for the detailed argument on the existence of the ODE system (\ref{qk-ode1}) and (\ref{qk-ode2}).

For the second part, choose some $\tau_1$ and positive $a(\tau_1)$ and $b(\tau_1)$ satisfying (\ref{qk-int}) and $y=\frac{a}{b}<r_2$. Then, the ODE (\ref{ode-y}) on $y$  implies that $\frac{\de y}{\de \tau}>0$. It is also easy to see that both $a(\tau)$ and $b(\tau)$ are increasing in $\tau$. The ODEs (\ref{qk-ode1}) and (\ref{qk-ode2}) can be solved for all $\tau>\tau_1$ since $y(\tau)< r_2$ for all $\tau>\tau_1$. By L'H\^opital's rule, it is easy to see from the ODEs (\ref{qk-ode1}) and (\ref{qk-ode2}) that
\[
\lim_{\tau\to \infty}y(\tau)=\lim_{\tau\to \infty}\frac{4+2mq^2 y^2}{2p-3q^2 y}.
 \]
 In view of $y(\tau)<r_2$, this further implies that $\lim_{\tau\to \infty} y(\tau)=r_2$. Hence, the re-scaled limit is $\widetilde{g}_{e_2}$. As $\tau$ decreases, $y(\tau)$ decreases.
Also, from the ODEs (\ref{qk-ode1}) and (\ref{qk-ode2}), either $a(\tau)$ or $b(\tau)$ will decrease to zero at $\tau_0$, for some
$\tau_0$. Again by L'H$\hat{o}$pital's rule, one can rule out the possibility that both $a(\tau)$ and $b(\tau)$ decrease to zero at $\tau_0$ simultaneously, since that would imply $y(\tau_0)$ equals either to $r_2$ or $r_1$ (which is impossible). Also by the fact that $0<y<r_2$ for $\tau\in (\tau_0, \tau_1)$, we rule out the possibility of $b(\tau_0)=0$ and $a(\tau_0)>0$. Hence, we conclude that at $\tau_0$, $a(\tau_0)=0$ and $b(\tau_0)>0$.

The type-I claim can be checked by the curvature formulae (\ref{principle-curvature}).
\end{proof}

\begin{remark} The existence of $\mathsf{SU}(2)$-principle bundle over a quaternion-K\"ahler manifold with a connection satisfying (\ref{qk-p-curv}) is not completely understood as in $\mathsf{U}(1)$-bundle case. Beyond the canonical associated one we do not know how to obtain other $\mathsf{SU}(2)$-principle bundles satisfying (\ref{qk-p-curv})   in general.
\end{remark}

Besides $\HP^m$, there are other quaternion-K\"ahler manifolds. There are infinitely many symmetric examples which can be found in \cite{Besse}, Table 14.52. They were classified by Wolf \cite{wolf-qk}. Hence there are many examples to which the above theorem can be applied.

\section{Ancient solutions via the Riemannian submersion}

Now we use the formulae and setting in Section 2.3 to construct the ancient solutions via the Riemannian submersion. We need this more general formulation particularly for constructing an ancient solution on the total space of the generalized Hopf fibration:
$\Sph^7\to \Sph^{15} \to \Sph^8$, since this does not fit into the formulation via  principle bundles.  Even though the formulation is quite general, in view of the rigidity result of Gromoll-Grove \cite{GG} and Wilking \cite{rigidity-w}, it is not as flexible as it appears if one insists that the fiber is a round sphere. The following proposition is the key step for our construction.

\begin{prop} Let $\pi: (P, g)\to (M, \check{g})$ be a Riemannian submersion with totally geodesic fiber.
 Let $g= \hat{g} + \check{g}$ be the metric decomposition. Suppose that the metrics on $P$, $M$ and  on the fibers are all Einstein with
 \begin{equation}\label{einstein}
  \Ric(g) = \lambda g, \quad \quad \Ric(\check{g}) = \check{\lambda} \check{g}, \quad \quad \Ric(\hat{g}) = \hat{\lambda} \hat{g}.
\end{equation}
 Let  $\widetilde{g}_{a, b}(\tau) = a(\tau) \hat{g} + b(\tau) \check{g}$.
 Then, $\widetilde{g}_{a, b}$ solving  the  Ricci flow equation
 is equivalent to
 \begin{eqnarray}
 & & \frac{ \de a}{\de \tau} =  2 \hat{\lambda} + 2 (\lambda -\hat{\lambda}) \frac{a^2}{b^2}, \label{sb-ode1}\\
 & & \frac{\de b}{\de \tau} = 2 \check{\lambda} - 2( \check{\lambda}- \lambda) \frac{a}{b} \label{sb-ode2}
 \end{eqnarray}
 whose first integral is given by
 \begin{equation}\label{sb-integral1}
 \left| 1- \frac{a}{b} \right|^{\frac{\lambda}{ \check{\lambda}-2 \hat{\lambda}}}
 \left| \frac{\hat{\lambda}}{ \check{\lambda}- \hat{\lambda}} -\frac{a}{b} \right|^{-\frac{\check{\lambda}^2 -2 \hat{\lambda} \check{\lambda}
 + \hat{\lambda} \lambda}{(\check{\lambda}-2 \hat{\lambda})(\check{\lambda}- \hat{\lambda})}} b^{-1} =
 \Lambda,
 \end{equation}
 where $\Lambda\ge 0$ is a constant.
 \end{prop}

 \proof{}
 By (\ref{sb-ricci-2}) and the assumption (\ref{einstein}), the Ricci tensor of $\widetilde{g}_{a, b} = a \hat{g} + b \check{g}$ is given by
 \[
 \Ric (\tilde{g}) = \left( \hat{\lambda} + (\lambda -\hat{\lambda}) \frac{a^2}{b^2} \right) \hat{g}
 + \left( \check{\lambda} -( \check{\lambda}- \lambda) \frac{a}{b} \right) \check{g}
 \]
 Let $y =  \frac{a}{b}$. Also, denote $y' = \frac{\de y}{\de \tau}$ and likewise for the derivatives $a'$ and $b'$.  Then, $y$ satisfies
 \beqn
 y' &=&  \frac{1}{b} \left( a'-  \frac{a}{b} b' \right) \CR
 &=& \frac{1}{b} \left(a'-  y b' \right) \CR
 &=& \frac{2}{b} \left(\hat{\lambda} + (\lambda -\hat{\lambda}) y^2 -  y (\check{\lambda} - ( \check{\lambda}- \lambda) y ) \right) \CR
 &=& \frac{2}{b} \left( \hat{\lambda} - \check{\lambda} y + (\check{\lambda} -\hat{\lambda} ) y^2\right).
 \eeqn
 Dividing the above by (\ref{sb-ode2}) we have that
 \beqn
 \frac{y'}{b'} = \frac{1}{b} \frac{\hat{\lambda} - \check{\lambda} y + (\check{\lambda} -\hat{\lambda} ) y^2}{\check{\lambda} -
  ( \check{\lambda}- \lambda)y}.
 \eeqn
 One can separate $y$ and $b$ such that
 \beqn
 \frac{\check{\lambda} -
  ( \check{\lambda}- \lambda)y}{\hat{\lambda} - \check{\lambda} y + (\check{\lambda} -\hat{\lambda} ) y^2} y' = \frac{b'}{b} .
 \eeqn
 Notice that
 \[
 \frac{\check{\lambda} -
  ( \check{\lambda}- \lambda)y}{\hat{\lambda} - \check{\lambda} y + (\check{\lambda} -\hat{\lambda} )
  y^2} = \frac{\lambda}{ \check{\lambda}-2 \hat{\lambda}} \frac{1}{y-1} -\frac{\check{\lambda}^2 -2 \hat{\lambda} \check{\lambda}
 + \hat{\lambda} \lambda}{(\check{\lambda}-2 \hat{\lambda})(\check{\lambda}- \hat{\lambda})} \frac{1}{y-\frac{\hat{\lambda}}{ \check{\lambda}- \hat{\lambda}}}.
 \]
 The first integral claimed in the proposition follows by integration of the separable equation on $y$ and $b$.
 \qed

By the equations found in Section 2.3, it is also easy to see that  $\check{\lambda}\ge\lambda\ge\hat{\lambda}$. When $\Lambda=0$,
we observe that if $\hat{\lambda}> 0$ and $2\hat{\lambda}\ne \check{\lambda}$,  there are two trivial solutions (Einstein metrics). (This is known.  See for example, Theorem 9.37 of \cite{Besse}.)  They correspond to $\frac{a}{b}=1$ and
$\frac{a}{b}= \Lambda_1\doteqdot \frac{\hat{\lambda}}{ \check{\lambda}- \hat{\lambda}}$, which are given by
\begin{eqnarray}
\de s^2_{e_1}(\tau) &=&2\lambda \tau g,\label{sb-ein1}\\
\de s^2_{e_2} (\tau) &=& 2\Lambda_2 \tau (\Lambda_1 \hat{g}+ \check{g}), \label{sb-ein2}
\end{eqnarray}
where $\Lambda_2=\frac{\check{\lambda}^2 -2 \hat{\lambda} \check{\lambda}+ \hat{\lambda} \lambda}{\check{\lambda}- \hat{\lambda}}.$ The first is the Einstein metric we started with. The second is a different  Einstein metric on the space.

\begin{theorem} \label{sb-main1} Assume that $\Lambda_1\ne 1$ and $\hat{\lambda}>0$. There exists an ancient solution $\de s^2(\tau)$ to Ricci flow on the total space $P$ with the  slope $y$ between $\Lambda_1$ and $1$.  If $\Lambda_1<1$, it flows (after re-normalization) the Einstein metric $\de s^2_{e_2}(\frac{1}{2\Lambda_1})$ into the Einstein metric $\de s^2_{e_1}(\frac{1}{2\lambda})$ as $t$ increases from $-\infty$ to some $t_0$. If $\Lambda_1>1$, it flows $\de s^2_{e_1}(\frac{1}{2\lambda})$ into $\de s^2_{e_2}(\frac{1}{2\Lambda_1})$ as $t$ increases from $-\infty$ to some $t'_0$. Both solutions are of type-I.
\end{theorem}
\begin{proof} We divide into two cases, $\Lambda_1 <1$ and $\Lambda_1>1$. Since for the application we always have  $\Lambda_1<1$ we shall prove this case and omit the details to the other (dual) case, whose proof is similar.

For $\Lambda_1 <y<1$, let
$$
F(y)\doteqdot  \left( 1- y \right)^{\frac{\lambda}{ \check{\lambda}-2 \hat{\lambda}}}
 \left( y-\frac{\hat{\lambda}}{ \check{\lambda}- \hat{\lambda}}  \right)^{-\frac{\check{\lambda}^2 -2 \hat{\lambda} \check{\lambda}
 + \hat{\lambda} \lambda}{(\check{\lambda}-2 \hat{\lambda})(\check{\lambda}- \hat{\lambda})}}= \left( 1- y \right)^{\frac{\lambda}{ \check{\lambda}-2 \hat{\lambda}}}
 \left( y-\Lambda_1  \right)^{-\frac{\Lambda_2}{\check{\lambda}-2 \hat{\lambda}}}.
$$
Under the assumption that $\Lambda_1<1$, it is easy to check that $\Lambda_1 \lambda < \Lambda_2 < \lambda$. Hence, $\frac{\de F}{\de y}<0$ for $y\in (\Lambda_1, 1)$.
Pick $\tau=1$ and $a(1)>0$ and $b(1)>0$ such that they satisfy $y=\frac{a}{b}\in (\Lambda_1, 1)$ and $$F(y)=b\Lambda
$$ for some $\Lambda$. By the previous proposition, we know that this condition will be preserved for solutions $a(\tau)$ and $b(\tau)$ of the Ricci flow equation (\ref{sb-ode1}) and (\ref{sb-ode2}).  It is easy to infer from the Ricci flow equations that
$$
\frac{\de a}{\de \tau} \ge 2\hat{\lambda}, \quad \quad \frac{\de b}{\de \tau} \ge 2\lambda.
$$
Hence, $a(\tau)$ and $b(\tau)$ are increasing in $\tau$, and, by the short time existence, one can solve (\ref{sb-ode1}) and (\ref{sb-ode2}) for some interval $(1-\delta, 1+\delta)$.
On the other hand
$$
\frac{ \de y}{ \de \tau} =\frac{2(\check{\lambda}-\hat{\lambda})}{b}(y-\Lambda_1)(y-1).
$$
Hence, for $y\in (\Lambda_1, 1)$ it is decreasing in $\tau$. This implies that as $\tau$ increase, the right hand side of (\ref{sb-ode1}) and (\ref{sb-ode2}) stay bounded above by a fixed number. This implies, by the existence theorem of ODE, say Theorem 7 of \cite{H}, that one can solve (\ref{sb-ode1}) and (\ref{sb-ode2}) for all $\tau>1$.
As $\tau$ decreases, one can solve the equations as long as $a$ and $b$ stay positive and $y<1$. By the uniqueness, it is not possible for $y$ to reach  $1$ while $a(\tau)$ and $b(\tau)$ remain positive. Assume that as $\tau \to \tau_0$, $a(\tau)\to 0$. Since $y(\tau) \ge \Lambda_1$, it also implies that $b(\tau)\to 0$. Using (\ref{sb-ode1}) and (\ref{sb-ode2}) we can compute that $y(\tau)\to 1$ as $\tau\to \tau_0$. Hence, if we blow up the metric $\de s^2(\tau)$ by $\frac{1}{\tau-\tau_0}$, it will approach to $\de s^2_{e_1}(\frac{1}{2\lambda})$. Similarly, we can argue that as $\tau\to \infty$, $y(\tau)\to \Lambda_1$. Hence, if we blow down the metric by $\frac{1}{\tau}$ as $\tau \to \infty$, $\de s^2(\tau)$ will approach to $\de s^2_{e_2}(\frac{1}{2\Lambda_2})$.

The type-I claim follows from the formulae in Section 2.3, the fact that $C_1\tau \le b(\tau)\le C_2 \tau$ for some positive constants $C_1, C_2$ and that $y$ stays bounded. \end{proof}

 Applying the  same argument in the above theorem we have  the following result.
\begin{corollary}\label{sb-collapsing}
 Assume that $\hat{\lambda}> 0$.  If $\Lambda_1<1$, there exists an ancient solution $\de s^2(\tau)$ to Ricci flow on the total space $P$, such that it exists for $t \in (-\infty, t_0)$ and with $y(\tau)<\Lambda_1$. Moreover, it is of type-I and as $t\to t_0$, $\de s^2(\tau)$ collapses into $b\, \check{g}$.
\end{corollary}

If $\Lambda_1>1$ and $y(\tau)$ starts with some amount bigger than $\Lambda_1$ at,  say $\tau=1$, it is not hard to see that the solution become singular for some $\tau_1>1$. So in this case, there exists no ancient solution.

\begin{remark} When $\hat{\lambda}=0$, there also exists ancient solution on the total space $P$ such that as $t\to$ some singular time, it converges (after re-scaling) to the original Einstein metric on $P$. The proof is similar to Theorem \ref{main1-u1}. Similarly, the ancient solution of this case is of type-I and collapsed.
\end{remark}

A good set of examples to which Theorem \ref{sb-main1} can be applied may be obtained by the fibration
$$
H/K\to G/K\to G/H: \quad gK \to gH
$$
 via Lie groups,  $K\subset H \subset G$, with $K$ and $H$ being compact subgroups of a compact Lie $G$. Some concrete examples are listed below.

 {\bf Example 1.} The Hopf fibration: $\Sph^3 \to \Sph^{4m+3} \to \HP^m$. This is given by $$\pi: (q_1, \cdots, q_{m+1})\to [q_1,  \cdots ,  q_{m+1} ] , $$ where $(q_1, \cdots, q_{m+1})\in \mathbb{H}^{m+1}$ with $\sum |q_i|^2 =1$. Endow $\Sph^{4m+3}$ with the constant curvature $1$ metric, and the symmetric metric $\check{g}$ on $\HP^{m}$ with sectional curvature between $1$ and $4$. Now
  $\lambda=4m+2, \hat{\lambda}=2$ and $\check{\lambda}=4m+8$ with $\Lambda_1<1$. The non-canonical Einstein metric was found first in \cite{J}. Its sectional curvature is positive and  has pinching constant $\frac{1}{(2m+3)^2}$. Theorem \ref{sb-main1} concludes that there exists an ancient solution which `connects' it with the canonical round  Einstein metric. However, due to the result of Tachibana \cite{Ta} and Wolf \cite{wolf}, the non-canonical Einstein metric can not have nonnegative curvature operator. Note that Theorem \ref{sb-main1} also conclude that the non-canonical Einstein metric from \cite{J} is a unstable fixed point of Ricci flow.

\begin{remark} The above case can also be derived from Theorem \ref{qk-main1}. Due to the ample examples of quaternion-K\"ahler manifolds, one expects that the cases in Theorem  \ref{qk-main1} are not completely contained by Theorem \ref{sb-main1}. For the overlapping cases, note that the `slope' function $y$ in Section 6 and Section 7 are different by a factor.
\end{remark}

  {\bf Example 2.}  Consider $\pi: \mathbb{C}P^{2m+1} \to \mathbb{H}P^m$ defined by
  \[
 [z_1, z_2, \cdots, z_{2m+1}, z_{2m+2}] \to [z_1+z_2{\bf j}, \cdots,
 z_{2m+1}+  z_{2m+2}{\bf j} ] .
  \]
 This is a fibre bundle with totally geodesic fibre $\Sph^2= \mathsf{Sp}(1)/\mathsf{U}(1)$. The
 Fubini-Study metric on $\mathbb{C}P^{2m+1}$ (with curvature between $1$ and
 $4$) induces  a metric of Fubini-Study type on
  $\mathbb{H}P^{m}$ with curvature ranging between $1$ and $4$. Both metrics are
 Einstein and $\lambda = 4m+4$, $\check{\lambda}=4m+8$. The metric on the fibre
 is of constant curvature $4$, so $\hat{\lambda}=4$. Clearly $\Lambda_1<1$. The existence of the non-canonical Einstein metric on $\CP^{2m+1}$ was found by Ziller \cite{Ziller1}. Its sectional curvature is positive with pinching constant $\frac{1}{4(m+1)^2}$. It is Hermitian (with respect to the usual complex structure) \cite{Ziller1}. Theorem \ref{sb-main1} concludes that there exists an ancient solution which `connects' it with the canonical Fubini-Study metric and Ziller's Einstein metric on $\CP^{2m+1}$ is a unstable fixed point of Ricci flow equation.

    {\bf Example 3.} Let $\mathbb{O}$ be octonion numbers. One can identify $\mathbb{R}^{16}$ with $\mathbb{O}^2$, and $\mathbb{R}^9$ with $\mathbb{R} \oplus \mathbb{O}$.
 The octonionic Hopf bundle is a $\Sph^7$-bundle over $\Sph^8$ defined by
  \[
  \pi: (o_1, o_2) \longrightarrow (|o_1|^2-|o_2|^2, 2 \bar{o}_1 o_2 ) ,
  \]
  where $(o_1, o_2)\in \mathbb{O}^2$ with $|o_1|^2+|o_2|^2 =1$.
  Consider on $\Sph^{15}$ the canonical metric with constant curvature $1$, hence
  $\lambda = 14$. The fibre $\Sph^7$ is totally geodesic with constant curvature $1$, so $\hat{\lambda}=6$. This metric induces a metric on $\Sph^8$ with constant curvature $4$, thus
  $\check{\lambda}=28$. Again $\Lambda_1<1$. The third non-canonical Einstein metric on $\Sph^{15}$ (besides the one in the example 1 above) was found in \cite{BK} by Bourguignon and Karcher (see also \cite{Besse}). It has pinching constant $\frac{9}{121}$. Theorem \ref{sb-main1} concludes that there exists another ancient solution on $\Sph^{15}$ `connecting' it with the canonical Einstein metric and implies that the non-canonical Einstein metric of Bourguignon-Karcher is a unstable fixed point. This together with results from Section 5, Example 1 and Corollary \ref{sb-collapsing} proves the theorem stated in the introduction.

{\bf Example 4.} Let $P$ be the twistor space over a compact quaternion-K\"ahler manifold $M^{4m}$ ($m\ge 2$) with positive scalar curvature. By a result of Salamon \cite{Sa} and B\'erard-Bergery \cite {BB2} (see also \cite{Besse}, Theorem 14.9), one can endow $P$ with a K\"ahler-Einstein metric such that the projection to $M$ is a Riemannian submersion with totally geodesic fibers. Since the Riemannian submersion does not  decrease the curvature, one can see that $\Lambda_1<1$ for this case too. One can check for the resulting K\"ahler-Einstein metric that $\lambda$, $\check{\lambda}$ and $\hat{\lambda}$ have the same values as in Example 2. Hence,  this can be viewed as a generalization of Example 2. Note that the other Einstein metric on the twistor space is a Hermitian non-K\"ahler manifold (cf. Corollary 14.84 of \cite{Besse}). Similarly, Theorem \ref{sb-main1} concludes that this Einstein metric is also unstable.

{\bf Example 5.} In dimension 6, 7, 12, 13, 24, Berger\cite{Berger-pos}, Wallach\cite{wallach-pos}, Aloff-Wallach \cite{AW-pos} constructed homogenous spaces of positive sectional curvature. It turns out that on these spaces, one can endow the Riemannian submersion structure satisfying Theorem \ref{sb-main1}. We refer to \cite{kong}, a forthcoming article of the second author for the details. Hence, there exist ancient solutions on these spaces too.

 The ancient solutions obtained from Theorem \ref{sb-main1} are all non-collapsed.
It is also easy to check that the standard  Einstein metric  has greater entropy (in the sense of Perelman \cite{P-entropy}) than the noncanonical Einstein metrics. Example 1 shows that in the classification result of \cite{Ni-anc}, the condition on the curvature can not be weaken to the positivity of the sectional curvature. Examples 1 and 3 also show that one can not expect that non-collapsed ancient solutions are rotationally symmetric even assuming the nonnegativity of the sectional curvature.

\section{Appendix-derivation of ODEs on Fateev's ansatz}

Here we follow the definitions, notations and computations made in Section 3. The main goal is to reduce the Ricci flow equation into a system of ODE, namely Proposition \ref{Fateev=2eqs}, under the ansatz in Section 3.
\subsection{}
Let $\Delta \doteqdot BC-D^2$. If we introduce $y_1=\theta, y_2=\chi_1, y_3=\chi_2$, the Christoffel symbols are given by
\begin{eqnarray*}
\left(\Gamma^1_{ij}\right) =\left(\begin{array}{lll}\frac{1}{2}\frac{A'}{A}& \quad 0& \quad 0 \\
0\quad & -\frac{1}{2}\frac{B'}{A}& -\frac{1}{2}\frac{D'}{A}\\
0 \quad& -\frac{1}{2} \frac{D'}{A}& -\frac{1}{2}\frac{C'}{A}\end{array}\right)
, \quad \quad
\left(\Gamma^2_{ij}\right) =\left(\begin{array}{lll}\frac{1}{A}& \quad 0& \quad 0 \\
0\quad & \,\,\, \frac{C}{\Delta}& -\frac{D}{\Delta}\\
0 \quad& -\frac{D}{\Delta} & \,\,\, \frac{B}{\Delta}\end{array}\right),
\end{eqnarray*}
\begin{eqnarray*}
\left(\Gamma^3_{ij}\right) =\left(\begin{array}{lll} \quad 0 \quad & \frac{1}{2\Delta}(CB'-D D')& \frac{1}{2\Delta}(C D'-D C') \\
\frac{1}{2\Delta}(CB'-D D')& \quad 0 \quad & \quad 0 \quad \\
\frac{1}{2\Delta}(C D'-D C') & \quad 0 \quad  & \quad 0 \quad\end{array}\right).
\end{eqnarray*}
Here, $A'=\frac{\partial A}{\partial \theta}$ and the same definition applies to $B', C', D'$. Using the formula
$$
R_{ij} =\frac{\partial \Gamma^t_{ij}}{\partial y_t}-\frac{\partial \Gamma^t_{it}}{\partial y_j}+\Gamma^s_{ij}\Gamma^t_{st} -\Gamma^s_{it}\Gamma_{sj}^t
$$
and the above expressions for the Christoffel symbols, direct computation yields
\begin{eqnarray}
R_{11} &=&-\left(\Gamma^2_{12}+\Gamma^3_{13}\right)'+\Gamma^1_{11}\Gamma^2_{12}+\Gamma^{1}_{11}\Gamma^3_{13}
-(\Gamma^2_{12})^2 -2\Gamma^2_{13}\Gamma^3_{12}-(\Gamma^3_{13})^2\nonumber \\
&=& \frac{A'}{4A}\frac{CB'+BC'-2DD'}{\Delta}-\frac{CB''+BC''-2D D''}{2\Delta}\label{r11}\\
&\,& +\frac{1}{4\Delta^2}\left[ C'^2B^2 +C^2 B'^2 +2 D^2 D'^2\right. \nonumber\\
&\,& \left. -4C'B D D' -4C B' D D' +2C'B'D^2 +2D'^2 BC\right];\nonumber
\end{eqnarray}
\begin{eqnarray}
R_{22} &=&\left(\Gamma^1_{22}\right)'+\Gamma^1_{22}\Gamma^1_{11}+\Gamma^{1}_{22}\Gamma^3_{13}
-\Gamma^1_{23}\Gamma^3_{12}-\Gamma^2_{21}\Gamma^1_{22}\nonumber \\
&=& \frac{A'B'}{4A^2}-\frac{B''}{2A}+\frac{1}{4A \Delta}
\left[B'(CB'-BC')+2D'(D'B-DB') \right];\label{r22}
\end{eqnarray}

\begin{eqnarray}
R_{23} &=&\left(\Gamma^1_{23}\right)'+\Gamma^1_{23}\Gamma^1_{11}-\Gamma^{1}_{22}\Gamma^2_{13}
-\Gamma^3_{21}\Gamma^1_{33}\nonumber \\
&=& \frac{A'D'}{4A^2}-\frac{D''}{2A}+\frac{1}{4A \Delta}
\left[B'(CD'-DC')+C'(BD'-DB') \right];\label{r23}
\end{eqnarray}

\begin{eqnarray}
R_{33} &=&\left(\Gamma^1_{33}\right)'+\Gamma^1_{33}\Gamma^1_{11}+\Gamma^1_{33}\Gamma^2_{12}+\Gamma^1_{33}\Gamma^3_{13}
-\Gamma^{1}_{32}\Gamma^2_{13}
-2\Gamma^1_{33}\Gamma^3_{13}\nonumber \\
&=& \frac{A'C'}{4A^2}-\frac{C''}{2A}+\frac{1}{4A \Delta}
\left[C'(C' B-C B')+2D'(C D'-D C') \right]\label{r33}
\end{eqnarray}
and the Ricci curvature has the form
$$
\left(R_{ij}\right)=\left(\begin{array}{lll}  R_{11}  &  0 & 0 \\
0\quad& R_{22} &  R_{23} \\
0 & R_{32}   & R_{33}\end{array}\right).
$$
Note that $R_{ij}$ can also be written in a more convenient form
\begin{eqnarray}
R_{11}&=&\frac{1}{4}\left(\log A\right)' \left(\log \Delta\right)'-\frac{1}{2}\left(\log \Delta \right)''-\frac{1}{4}\left[\left(\log \Delta \right)'\right]^2 +\frac{B'C'-D'^2}{2\Delta},\label{r11-clean}\\
&=&\frac{\Delta'}{4\Delta}\left(\frac{A'}{A}+\frac{\Delta'}{\Delta}\right)-\frac{\Delta''}{2\Delta}+
\frac{B'C'-D'^2}{2\Delta}\nonumber\\
R_{22}&=& \frac{B'}{4A}\left(\frac{A'}{A}+\frac{\Delta'}{\Delta}\right)-\frac{B''}{2A}-\frac{B}{A}
\frac{B'C'-D'^2}{2\Delta},  \label{r22-clean}\\
R_{23}&=& \frac{D'}{4A}\left(\frac{A'}{A}+\frac{\Delta'}{\Delta}\right)-\frac{D''}{2A}-\frac{D}{A}
\frac{B'C'-D'^2}{2\Delta},  \label{r23-clean}\\
R_{33}&=&\frac{C'}{4A}\left(\frac{A'}{A}+\frac{\Delta'}{\Delta}\right)-
\frac{C''}{2A}-\frac{C}{A}\frac{B'C'-D'^2}{2\Delta}.   \label{r33-clean}
\end{eqnarray}

Introducing the {\it integrability conditions}
\begin{equation}\label{ic}
(u+d)^2=a^2+c^2, \quad \quad d^2=b^2 +c^2
\end{equation}
the expressions for $\Delta$ and $\bar{\Delta}$ simplify to
\begin{equation}\label{ic-c1}
\bar{\Delta}=\bB \bC -\bD^2 =\frac{1}{4}\sin ^2 2\theta \, w(\tau, \theta), ~ \mbox{ hence }\Delta =\frac{1}{4}\frac{\sin ^2 2\theta}{w(\tau, \theta)}.
\end{equation}
Straightforward computation also shows that
\begin{eqnarray}\label{dd}
\frac{B'C'-D'^2}{2\Delta}&=& \frac{\bB'\bC'-\bD'^2}{2\bar{\Delta}}+\frac{1}{2}\left(\frac{w'}{w}\right)^2 -\frac{1}{2}\frac{w'}{w}\frac{\bar{\Delta}'}{\bar{\Delta}} \nonumber\\
&=& \frac{\bB'\bC'-\bD'^2}{2\bar{\Delta}}-2\frac{\cos 2\theta}{\sin 2\theta} \frac{w'}{w} , \\
\frac{\bB'\bC'-\bD'^2}{2\bar{\Delta}}&=&-\frac{2}{w}\left((u+2d)^2-4b^2 \cos^2 2\theta\right)=-\frac{2}{w}(u+2d)^2+2\frac{\cos 2\theta}{\sin 2\theta} \frac{w'}{w} , \nonumber
\end{eqnarray}
since
\begin{eqnarray*}
\left(\log A\right)'&=& -\frac{w'}{w}\\
\left(\log \Delta\right)'&=& 4\frac{\cos 2\theta}{\sin 2\theta}-\frac{w'}{w}\\
\left(\log \Delta \right)''&=& -\frac{8}{\sin^2 2\theta}-\frac{w''}{w}+\left(\frac{w'}{w}\right)^2.
\end{eqnarray*}
Note that (\ref{dd}) also implies that
\begin{equation}\label{this}
\frac{B'C'-D'^2}{2\Delta}=-\frac{2(u+2d)^2}{w}.
\end{equation}
Putting all together we have that
\begin{eqnarray}
R_{11}&=& 4+\frac{1}{2}\frac{w''}{w}-\frac{1}{2}\left(\frac{w'}{w}\right)^2 -\frac{w'}{w}\frac{\cos 2\theta}{\sin 2\theta}
+\frac{\bB'\bC'-\bD'^2}{2\bar{\Delta}}\nonumber\\
&=&-\frac{2(u+2d)^2}{w}+4(a^2-b^2)\cdot \frac{ a^2 +b^2\cos^2 2\theta}{w^2}. \label{r11-clean1}
\end{eqnarray}
The first equation $\frac{\partial A}{\partial \tau }=\frac{1}{2}R_{11}$ gives rise to the following system of
three equations:
\begin{eqnarray}
\frac{\de \,u}{\de \tau}&=&-(u+2d)^2, \label{rf-11-1}\\
u\frac{\de\, a}{\de \tau}&=& -a(a^2-b^2),\label{rf-11-2}\\
u\frac{\de \,  b}{\de \tau}&=& b(a^2-b^2).\label{rf-11-3}
\end{eqnarray}

\subsection{} Next, we check the other three equations in the Ricci flow equation. Using the equations (\ref{ic}), (\ref{ic-c1}),
as before, $R_{22}, R_{23}, R_{33}$ can  be written in a more symmetric manner
\begin{eqnarray}
\frac{u}{\bB} R_{22}&=& -\frac{\bB''}{2\bB}+\frac{\bB'}{\bB}\left( \frac{\cos 2\theta}{\sin 2\theta}+\frac{w'}{2w}\right)
  +\frac{1}{2}\frac{w''}{w}-\frac{1}{2}\left(\frac{w'}{w}\right)^2 -\frac{w'}{w}\frac{\cos 2\theta}{\sin 2\theta}-\frac{B'C'-D'^2}{2\Delta}\nonumber\\
&=&R_{11}-4-2\frac{\cos 2\theta}{\sin 2\theta}\frac{w'}{w}+4\frac{(u+2d)^2}{w}
- \frac{\bB''}{2\bB}+\frac{\bB'}{\bB}\left(\frac{\cos 2\theta}{\sin 2\theta}+\frac{w'}{2w}\right),\label{r22-sys} \\
\frac{u}{\bD}R_{23}&=&-\frac{\bD''}{2\bD}+\frac{\bD'}{\bD}\left(\frac{\cos 2\theta}{\sin 2\theta}+\frac{w'}{2w}\right)
+\frac{1}{2}\frac{w''}{w}-\frac{1}{2}\left(\frac{w'}{w}\right)^2 -\frac{w'}{w}\frac{\cos 2\theta}{\sin 2\theta}-\frac{B'C'-D'^2}{2\Delta}\nonumber\\
&=&R_{11}-4-2\frac{\cos 2\theta}{\sin 2\theta}\frac{w'}{w}+4\frac{(u+2d)^2}{w}
- \frac{\bD''}{2\bD}+\frac{\bD'}{\bD}\left(\frac{\cos 2\theta}{\sin 2\theta}+\frac{w'}{2w}\right), \label{r23-sys}\\
\frac{u}{\bC} R_{33}&=&-\frac{\bC''}{2\bC}+\frac{\bC'}{\bC}\left( \frac{\cos 2\theta}{\sin 2\theta}+\frac{w'}{2w}\right)
 +\frac{1}{2}\frac{w''}{w}-\frac{1}{2}\left(\frac{w'}{w}\right)^2 -\frac{w'}{w}\frac{\cos 2\theta}{\sin 2\theta}-\frac{B'C'-D'^2}{2\Delta} \nonumber\\
&=&R_{11}-4-2\frac{\cos 2\theta}{\sin 2\theta}\frac{w'}{w}+4\frac{(u+2d)^2}{w}
- \frac{\bC''}{2\bC}+\frac{\bC'}{\bC}\left( \frac{\cos 2\theta}{\sin 2\theta}+\frac{w'}{2w}\right).\label{r33-sys}
\end{eqnarray}

Using (\ref{this}), straightforward computation shows that the equation
 $\frac{\partial D}{\partial \tau} =\frac{1}{2}R_{23}$ is equivalent to
 \begin{eqnarray*}
 2\left(\frac{u}{w}\right)_\tau \frac{\bD}{u}+2\frac{u}{w}\left(\frac{\bD}{u}\right)_\tau &=& -\frac{\bD''}{2u}+\frac{\bD'}{4u}\left(4 \frac{\cos 2\theta}{\sin 2\theta}+2\frac{w'}{w}\right)\\
&\quad& +\frac{\bD}{u}\left(R_{11}-4-2\frac{\cos 2\theta\, w'}{\sin 2\theta\, w}+4\frac{(u+2d)^2}{w}\right).
 \end{eqnarray*}
 Here $(\cdot)_\tau$ means the derivative with respect to $\tau$.
 Using $2\frac{\partial\, A}{\partial \tau} =R_{11}$, and multiplying both sides of the equation by $u\cdot w$, this is further reduced to
 $$
 2(\bD_\tau \, u -u_\tau \, \bD)=-\frac{\bD''\, w}{2}+\frac{\bD'\, w}{4}\left(4 \frac{\cos 2\theta}{\sin 2\theta}+2\frac{w'}{w}\right)-\bD (4a^2+4b^2 \cos^2 2\theta)+4\bD (u+2d)^2.
 $$
 Since the first three terms on the right hand side add up to $0$, we have
 $$
 \left(c_\tau \, u -u_\tau\, c\right)\sin^2 2\theta =2c(u+2d)^2\sin^2 2\theta,
 $$
 which by (\ref{rf-11-1})  reduces to
 \begin{equation}\label{rf-23-1}
 \frac{\de\, c}{\de \tau}=-\frac{\de\, u}{\de \tau} \frac{c}{u}
 \end{equation}
or simply put $\frac{\de}{\de \tau} (u\, c) =0$, which is equivalent to $u(\tau)=\frac{\Lambda_1}{c(\tau)}$.

Before going further, it is helpful to note that $2\frac{\partial A}{\partial \tau} =R_{11}$ and $2\frac{\partial D}{\partial \tau}  =R_{23}$ are equivalent to the equations (\ref{rf-11-1}), (\ref{rf-11-2}),
(\ref{rf-11-3}) and (\ref{rf-23-1}). However, these four equations can be put into an even simpler form
\begin{equation}\label{rf-all}
\frac{\de\, u}{\de \tau}=-(u+2d)^2, \quad \frac{\de (u\, c)}{\de \tau}=0, \quad  \frac{\de (a\, b)}{\de \tau}=0
\end{equation}
which,  as we shall show, are equivalent to the Ricci flow equation (\ref{rf1}). Indeed, assuming the relations (\ref{rf-all}),
if we introduce a new function
$v\doteqdot u+2d$, and write $u\, c=\Lambda_1$, $a\, b =\Lambda_2$, a simple calculation, making use of the integrability conditions (\ref{ic}),  shows that
$$
\left[ \frac{(v-u)^2}{4}-\frac{\Lambda_1^2}{u^2}\right]\left[ \frac{(v+u)^2}{4}-\frac{\Lambda_1^2}{u^2}\right]=\Lambda_2^2.
$$
This, in turn, implies that
$$
v=\sqrt{\frac{u^4-4\sqrt{\Lambda_1^2+\Lambda_2^2} u^2 +4\Lambda_1^2}{u^2}}.
$$
Taking the derivative with respect to $\tau$ on both sides of the equation $a^2-\frac{\Lambda^2_2}{a^2}=a^2-b^2=u\, v$
and substituting the expression for $v$ just computed, we obtain
$$
2\frac{a_\tau}{a} \left(a^2+b^2\right)=-v\left(2u^2 -4\sqrt{\Lambda_1^2+\Lambda_2^2}\right).
$$
On the other hand, since $a^2=\frac{(v-u)^2}{4} -\frac{\Lambda^2_1}{u^2}$ and $ b^2=\frac{(v+u)^2}{4}-\frac{\Lambda_1^2}{u^2}$, the right hand side above is just $-v(a^2+b^2)$. Hence, we arrive at
$$
\frac{a_\tau}{a}=-v=-\frac{a^2-b^2}{u} ,
$$
which is (\ref{rf-11-2}). From this (\ref{rf-11-3}) follows easily.

Finally, we verify that $2\frac{\partial B}{\partial \tau}=R_{22}$ and $2\frac{\partial C}{\partial \tau}=R_{33}$. We only check the first one,
since the second is exactly the same. From (\ref{r22-sys}), making use of equation $2\frac{\partial A}{\partial \tau}=R_{11}$ in the same way as above,  $2\frac{\partial B}{\partial \tau}=R_{22}$ reduces to
$$
 2(\bB_\tau \, u -u_\tau \, \bB)=-\frac{\bB''\, w}{2}+\frac{\bB'\, w}{4}\left(4 \frac{\cos 2\theta}{\sin 2\theta}+2\frac{w'}{w}\right)-\bB (4a^2+4b^2 \cos^2 2\theta)+4\bB (u+2d)^2.
 $$
By (\ref{rf-11-1}) it can be further reduced to
 $$
 2\bB_\tau \, u=-\frac{\bB''\, w}{2}+\frac{\bB'\, w}{4}\left(4 \frac{\cos 2\theta}{\sin 2\theta}+2\frac{w'}{w}\right)-\bB (4a^2+4b^2 \cos^2 2\theta)+2\bB v^2.
 $$
 Recall that $v=u+2d$ and observe that $\bB =\frac{1}{2}\left(v (\cos 2\theta +1)-2\sin^2\theta\right)$,
$\bB'=-\sin 2\theta(v+2d\cos 2\theta)$ and $\bB''= -2\cos 2\theta(v+2d\cos2\theta)+4d\sin^2 2\theta$. Straightforward computation shows that the right hand side of the equation above becomes
$$
v^3-2v(a^2+b^2) +(v^3-2v(a^2+b^2))\cos 2\theta +(2vb^2-v^2 d) \sin^2 2\theta.
$$
Hence $2\frac{\partial B}{\partial \tau}=R_{22}$ gives rise to a system of two equations
\begin{eqnarray}
 u \, v_\tau &=& v^3-2v(a^2+b^2), \label{rf-22-1}\\
u\, d_\tau &=& v^2 d -2 v b^2. \label{rf-22-2}
\end{eqnarray}
Similar calculation shows that the equation $2\frac{\partial C}{\partial \tau}=R_{33}$ can also be reduced to the above set of equations. It is
now a simple matter of direct checking to show that (\ref{rf-22-1}) and (\ref{rf-22-2}) follow from the equations (\ref{rf-all}). For example
\begin{eqnarray*}
u\, v_\tau &=& u \left(\frac{a^2-b^2}{u}\right)_\tau \\
&=& (a^2-b^2)_\tau -\frac{a^2-b^2}{u}u_\tau\\
&=& -2(a^2+b^2)v +v^3.
\end{eqnarray*}
Here, we make use of (\ref{rf-11-1})--(\ref{rf-11-3}), which are consequences of (\ref{rf-all}), to go from the second to the third line. Similarly,  (\ref{rf-22-2}) also follows from (\ref{rf-all}).

Summarizing the above, we showed the following.

\begin{proposition}\label{Fateev=2eqs}  Under the ansatz (\ref{ABCD}), the Ricci flow equation (\ref{rf1}) is equivalent to the three equations in (\ref{rf-all}).
\end{proposition}

\subsection{}
Finally, solving (\ref{rf-all}) in order to obtain the explicit formulae (\ref{formulae}) requires a skillful maneuver of  change of variables. This was done on pages 522--523 of \cite{Fa96}. We include the details below for completeness of our presentation. Let $u_1^2$ and $u_2^2$ be given by
$$
u_{1}^2= 2(\sqrt{\Lambda_1^2+\Lambda_2^2}+\Lambda_2), \quad u_{2}^2= 2(\sqrt{\Lambda_1^2+\Lambda_2^2}-\Lambda_2).
$$
Then, $v^2=\frac{(u^2-u_1^2)(u^2-u_2^2)}{u^2}$ and therefore
$$
\frac{\de\, u}{\de \tau} =-\frac{(u^2-u_1^2)(u^2-u_2^2)}{u^2}.
$$
The parameters $\nu$ and $k$ are introduced by writing
$$
u_1=\frac{\nu}{1-k^2}, \quad u_2=\frac{\nu k}{1-k^2} ,
$$
which also amount to
$$
4\Lambda_2=\frac{\nu^2}{1-k^2}, \quad 2\Lambda_1 =\frac{\nu^2 k}{(1-k^2)^2}.
$$
The ODE for $u$ is separable and can be solved by writing
\begin{equation}\label{1st-int}
-\frac{\nu}{1-k^2}\int \frac{1}{u^2-u_1^2}\, du +\frac{\nu k^2}{1-k^2}\int \frac{1}{u^2-u_2^2}\, du =\tau \nu
\end{equation}
 and letting
 $$
 \xi=\coth ^{-1}\left( \frac{(1-k^2)u}{\nu}\right)=-\frac{\nu}{1-k^2}\int \frac{1}{u^2-u_1^2}\, du.
 $$
Hence, one arrives at
   (\ref{eq-xitau}) by integrating the second term in (\ref{1st-int}), namely
$$
\nu \tau =f(\xi) \doteqdot \xi -\frac{k}{2}\log \left(\frac{1+k\tanh \xi}{1-k \tanh \xi}\right).
$$
Note that
$$
\frac{\de\, f}{\de \xi}=\frac{(1-k^2) \cosh ^2 \xi}{\cosh^2 \xi-k^2\sinh ^2 \xi} ,
$$
which approaches to $1-k^2$ as $\xi\to 0$ and approaches to $1$ as $\xi \to \infty$.
  Now $u=\frac{\nu}{1-k^2} \coth \xi$. The explicit formulae for the functions $a, b, d, c$ can be derived from the expression for $u$.

Finally we remark on another computation leading to Proposition \ref{Fateev=2eqs}. First note that (\ref{r22-clean})--(\ref{r33-clean}) can be written as
\begin{eqnarray}
\frac{A}{B}R_{22}&=& \frac{B'}{4B}\left(\frac{A'}{A}+\frac{\Delta'}{\Delta}-2\frac{B'}{B}\right)
-\frac{1}{2}\left(\frac{B'}{B}\right)'-
\frac{B'C'-D'^2}{2\Delta},  \label{r22-clean1}\\
\frac{A}{D}R_{23}&=& \frac{D'}{4D}\left(\frac{A'}{A}+\frac{\Delta'}{\Delta}
-2\frac{D'}{D}\right)-\frac{1}{2}\left(\frac{D'}{D}\right)'-
\frac{B'C'-D'^2}{2\Delta}, \label{r23-clean1}\\
\frac{A}{C}R_{33}&=& \frac{C'}{4C}\left(\frac{A'}{A}+\frac{\Delta'}{\Delta}-2\frac{C'}{C}\right)
-\frac{1}{2}\left(\frac{C'}{C}\right)'-
\frac{B'C'-D'^2}{2\Delta}.   \label{r33-clean1}
\end{eqnarray}
Then direct computation shows that
\begin{eqnarray}
&\,&\frac{u}{B}R_{22}= 2(u+2d)^2+4(a^2-b^2) \frac{ a^2 +b^2\cos^2 2\theta}{w}-4(u+2d) \frac{a^2+b^2\cos 2\theta}{u+d +d \cos2\theta}, \nonumber\\
&\,&\frac{u}{D}R_{23}= 2(u+2d)^2+4(a^2-b^2) \frac{ a^2 +b^2\cos^2 2\theta}{w}, \nonumber\\
&\,&\frac{u}{C}R_{33}=  2(u+2d)^2+4(a^2-b^2) \frac{ a^2 +b^2\cos^2 2\theta}{w}-4(u+2d) \frac{a^2-b^2\cos 2\theta}{u+d -d \cos2\theta}.  \nonumber
\end{eqnarray}
From this together with (\ref{r11-clean1}), one can also derive Proposition \ref{Fateev=2eqs}
\section*{Acknowledgments.} { SLK and LN would like to thank  Nolan Wallach for helpful discussions, Toti  Daskalopoulos, Richard Hamilton and Tom Ilmanen for their interests to this work. LN's research is  supported by NSF grant DMS-0805834.
 }

\bibliographystyle{amsalpha}

{\sc  Addresses}

{\sc Ioannis Bakas},
Department of Physics, University of Patras, 26500 Patras, Greece

email: bakas@ajax.physics.upatras.gr

\medskip

{\sc Shengli Kong},
Department of Mathematics, University of Science and Technology of China, Hefei, Anhui 230026, CHINA

email: kongs@ustc.edu.cn

\medskip

{\sc Lei Ni},
 Department of Mathematics, University of California at San Diego, La Jolla, CA 92093, USA


email: lni@math.ucsd.edu

\end{document}